\documentclass[a4paper,11pt]{amsart}

\usepackage{amssymb}
\usepackage{amstext}
\usepackage{amsmath}
\usepackage{amscd}
\usepackage{amsthm}
\usepackage{amsfonts}
\usepackage{enumerate}
\usepackage{graphicx}
\usepackage{latexsym}
\usepackage[all]{xy}
\usepackage[usenames]{color}

\newtheorem*{corollary*}{Corollary}
\newtheorem{theorem}{Theorem}[section]
\newtheorem{corollary}[theorem]{Corollary}
\newtheorem{lemma}[theorem]{Lemma}
\newtheorem{proposition}[theorem]{Proposition}
\newtheorem*{proposition*}{Proposition}

\theoremstyle{definition}
\newtheorem{definition}[theorem]{Definition}

\newtheorem{example}[theorem]{Example}
\newtheorem*{notation}{Notation}

\newtheorem*{set-up}{Set-up}

\theoremstyle{remark}
\newtheorem{assumption}[theorem]{Assumption}

\numberwithin{equation}{section}
\renewcommand{\mod}{\mathsf{mod}}
\newcommand{\proj}{\mathsf{proj}}

\newcommand{\End}{\operatorname{End}}
\newcommand{\Hom}{\operatorname{Hom}}
\newcommand{\add}{\operatorname{\mathsf{add}}}

\newcommand{\T}{\mathcal{T}}
\newcommand{\Db}{\mathsf{D}^{\rm b}}
\newcommand{\Kb}{\mathsf{K}^{\rm b}}

\newcommand{\type}[1]{\mathbb{#1}}

\newcommand{\tilt}{\operatorname{tilt}}

\newcommand{\silt}{\operatorname{silt}}
\newcommand{\tsilt}{\operatorname{2-silt}}
\newcommand{\ttilt}{\operatorname{2-tilt}}
\newcommand{\sttilt}{\mbox{\rm s$\tau$-tilt}\hspace{.01in}}
\newcommand{\La}{\Lambda}

\newcommand{\inc}{\mathbf{i}}
\newcommand{\id}{{\rm id}}

\newcommand{\I}{\widetilde{I}}
\newcommand{\wLa}{\widetilde{\Lambda}}
\newcommand{\RHom}{\mathbf{R}\strut\kern-.2em\operatorname{Hom}\nolimits}
\newcommand{\REnd}{\mathbf{R}\strut\kern-.2em\operatorname{End}\nolimits}

\newcommand{\Lwotimes}{\mathop{{\otimes}^\mathbf{L}_{\widetilde{\Lambda}}}\nolimits}

\newcommand{\bmu}{{\boldsymbol \mu}}

\begin{document}
\title[tilting complexes over preprojective algebras
of Dynkin type]
{Classifying tilting complexes over preprojective algebras
of Dynkin type}

\author{Takuma Aihara}
\address{Department of Mathematics, Tokyo Gakugei University, 4-1-1 Nukuikita-machi, Koganei,Tokyo 184-8501, Japan}
\email{aihara@u-gakugei.ac.jp}
\author{Yuya Mizuno}
\thanks{The first author was partly supported by IAR Research Project,
Institute for Advanced Research, Nagoya University and by Grant-in-Aid for Young Scientists 15K17516. The second author was supported by Grant-in-Aid for Young Scientists 26800009.}
\address{Graduate School of Mathematics, Nagoya University, Furocho, Chikusaku, Nagoya 464-8602, Japan}
\email{yuya.mizuno@math.nagoya-u.ac.jp}

\begin{abstract}
We study tilting complexes over preprojective algebras of Dynkin type. 
We classify all tilting complexes by 
giving a bijection between tilting complexes and the braid group of the corresponding folded graph. 
In particular, we determine the derived equivalence class of the algebra. 
For the results, we develop the theory of silting-discrete triangulated categories and give a criterion of silting-discreteness. 
\end{abstract}
\maketitle
\tableofcontents

%%%%%%%%%%%%%%%%%%%%%%%%%%%%%%%%%%%%%%%%%%%%%%%%%%%%%%%%%%%%%%%%%%%%%%%%%%%
%%%%%%%%%%%l%%%%%%%%%%%%%%%%%%%%%%%%%%%%%%%%%%%%%%%%%%%%%%%%%%%%%%%%%%%%%%%%
\section{Introduction}

\subsection{Background and motivation} 

Derived categories are nowadays considered as a fundamental object in many branches of mathematics including representation theory and algebraic geometry. 
Among others, one of the most important problems is to understand their equivalences. Derived equivalences provide a lot of interesting connections between various different objects and they are also quite useful to study structures of the categories.

It is known that derived equivalences are controlled by tilting objects (complexes) \cite{Ric,K} and therefore these constructions have been extensively studied. 
As a tool for studying tilting objects, 
Keller-Vossieck introduced the notion of silting objects (Definition \ref{def silt obj}), which is a generalization of tilting objects \cite{KV}. 
After that, it was shown that their mutation properties are much better than tilting ones and they yield a nice combinatorial description \cite{AI} (see Definition \ref{defsm}). 
Furthermore, silting objects have turned out to have deep connections with several important objects such as cluster tilting objects and $t$-structures, for example \cite{AIR,BRT,KY,BY,IJY,QW,BPP}.

One of the aim of the paper is to give a further development of the mutation theory of silting objects. 
In particular, we study a criterion when a triangulated category is \emph{silting-discrete} (Definition \ref{definition of silting-discrete}). 
A remarkable property of this class is that all silting objects are connected to each other by iterated mutation and this fact admits us to achieve a comprehensive understanding of the categories. 

Another aim of the paper is, by applying this technique, to classify all tilting complexes of \emph{preprojective algebras} of Dynkin type. 
Since preprojective algebras were introduced in \cite{GP,DR,BGL}, it turned out that they have
fundamental importance in representation theory as well as algebraic and differential geometry.
We refer to \cite{Rin} for quiver representations, \cite{L1,L2,KaS} for quantum groups,
\cite{AuR,CB} for Kleinian singularities, \cite{N1,N2,N3} for quiver varieties,
and \cite{GLS1,GLS2} for cluster algebras.

For the case of proprojective algebras of non-Dynkin type, 
its tilting theory  has been extensively studied in \cite{BIRS,IR1}. 
In particular, they show that certain ideals parameterized by the Coxeter group (see Theorem \ref{tau-weyl})  give tilting modules over the proprojective algebra and this fact provides a method for studying the derived category. 
On the other hand, in the case of Dynkin type, they are no longer tilting modules. 
Moreover, there is no spherical objects in this case and a similar nice theory had never been observed. 
In this paper, via a new strategy, we succeed to classify all tilting complexes as below.

\subsection{Our results} 
To explain our results, we give the following set-up.
Let $\Delta$ be a Dynkin graph and $\La$ the preprojective algebra of $\Delta$. 

First we study two-term tilting complexes of $\La$. 
For this purpose, we use \emph{$\tau$-tilting theory}. 
In \cite{M1}, the second author showed that the above ideals are support $\tau$-tilting $\La$-modules (Theorem \ref{tau-weyl}).
Then, combining the results of \cite{AIR}, we obtain a bijection between two-term silting complexes  of $\La$ and the Weyl group (Theorem \ref{tau-weyl}). 
Moreover we analyze this connection in more details and we can give a classification of two-term tilting complexes of $\La$ using the \emph{folded graph} $\Delta^{\rm f}$ of $\Delta$ (Definition \ref{def fold}) given by the following correspondences.

\[\begin{array}{|c|c|c|c|c|c|c|c|}
\hline                                                                       
\Delta & \type{A}_{2n-1},\type{A}_{2n}&\type{D}_{2n} &\type{D}_{2n+1} & \type{E}_6  & \type{E}_7 & \type{E}_8 \\ \hline
\Delta^{\rm f} &\type{B}_n  &\type{D}_{2n}        &\type{B}_{2n}      & \type{F}_4 & \type{E}_7  & \type{E}_8 \\ \hline
\end{array}\]

Then our first result is summarized as follows.

\begin{theorem}[Theorem \ref{action}]\label{main1} 
Let $W_{\Delta^{\rm f}}$ be the Weyl group of $\Delta^{\rm f}$ and $\ttilt\Lambda$ the set of isomorphism classes of basic two-term tilting complexes of $\Kb(\proj\La)$. 
Then we have a bijection 
$$W_{\Delta^{\rm f}}\longleftrightarrow\ttilt\Lambda.$$ 
\end{theorem}

We remark that we can give not only a bijection but also an explicit description of all two-term tilting complexes (Theorem \ref{tau-weyl}). 
On the other hand, we study an important relationship between two-term silting complexes and silting-discrete categories.
More precisely, we give the following criterion of silting-discreteness (tilting-discreteness). 

\begin{theorem}[Theorem \ref{silting-discrete}, Corollary \ref{tilting-discrete}]\label{main2}
Let $A$ be a finite dimensional algebra (respectively, finite dimensional selfinjective algebra). 
The following are equivalent.
\begin{itemize}
\item[(a)] $\Kb(\proj A)$ is silting-discrete (respectively, tilting-discrete). 
\item[(b)] $\tsilt_PA$ (respectively, $\ttilt_PA$) is a finite set for any silting (respectively, tilting) complex $P$. 
\item[(c)] $\tsilt_PA$ (respectively, $\ttilt_PA$) is a finite set for any silting (respectively, tilting) complex $P$ which is given by iterated irreducible left silting (respectively, tilting) mutation from $A$.
\end{itemize}
\end{theorem}

Here $\tsilt_PA$ (respectively, $\ttilt_PA$) denotes the subset of silting (respectively, tilting) objects $T$ in $\Kb(\proj A)$ such that $P\geq T\geq P[1]$ (Definition \ref{definition of silting-discrete}). 
An advantage of this theorem is that we can understand the condition of the all 
silting (respectively, tilting) objects by studying a certain special class of silting (respectively, tilting) objects. 
Then, we can apply Theorem \ref{main2} and obtain the following result.

\begin{theorem}[Theorem \ref{tilt discrete}, Proposition \ref{endo silt}]\label{main3} 
The endomorphism algebra of any irreducible left tilting mutation (Definition \ref{defsm}) of $\La$ is isomorphic to $\La$. 
In particular, the condition (b) of Theorem \ref{main2} is satisfied and hence $\Kb(\proj\La)$ is tilting-discrete. 
\end{theorem}

Then Theorem \ref{main3} implies that any tilting complexes are obtained from $\La$ by iterated irreducible mutation.
As a consequence of this result, we determine the derived equivalence class of $\La$ as follows.

\begin{corollary}[Theorem \ref{tilt discrete}]\label{main cor} 
Any basic tilting complex $T$ of $\La$ satisfies $\End_{\Kb(\proj\La)}(T)\cong\La$. In particular, the derived equivalence class coincides with the Morita equivalence class.
\end{corollary}

In fact, we give a more detailed description about tilting complexes. Indeed, using 
Theorem \ref{main1} and Corollary \ref{main cor}, we can show that irreducible tilting mutation satisfy braid relations (Proposition \ref{braid relations}), which provide a nice relationship between the braid group and tilting complexes (c.f. \cite{BT,ST,G,KhS}).

Recall that the braid group $B_{\Delta^{\rm f}}$ is defined by generators $a_i$ $(i\in \Delta^{\rm f}_0)$ with relations $(a_ia_j)^{m(i,j)}=1$ for $i\neq j$ (see subsection \ref{relation} for $m(i,j)$), that is, the difference with $W_{\Delta^{\rm f}}$ is that we do not require the relations $a_i^2=1$ for $i\in \Delta^{\rm f}_0$.  
We denote by $\bmu_{i}^{+}$ (respectively, $\bmu_{i}^{-}$) the irreducible left (respectively, right) tilting mutation associated with $i\in \Delta_0^{\rm f}$.

Then we can define the map from the braid group to tilting complexes and it gives a classification of tilting complexes as follows.

\begin{theorem}[Theorem \ref{braid action}]\label{main4}
Let $B_{\Delta^{\rm f}}$ be the braid group of $\Delta^{\rm f}$ and $\tilt\La$ the set of isomorphism classes of basic tilting complexes of $\La$. Then we have a bijection 
$$B_{\Delta^{\rm f}}\longrightarrow\tilt\La,$$ 
$$\ \ \ \ \ \ \ \ \ \ \ a=a_{i_1}^{\epsilon_{i_1}}\cdots a_{i_k}^{\epsilon_{i_k}}\mapsto
\bmu_a(\La):=\bmu_{i_1}^{\epsilon_{i_1}}\circ\cdots \circ \bmu_{i_k}^{\epsilon_{i_k}}(\La).$$
\end{theorem}

We now describe the organization of this paper.

In section 2, we deal with triangulated categories and study some properties of silting-discrete categories. In particular, we give a criterion of silting-discreteness. 
We also investigate a Bongartz-type lemma for silting objects.  
In section 3, we recall definitions and some results related to preprojective algebras.  
In section 4, we explain a connection between two-term silting complexes and the Weyl group. In particular, we characterize two-term tilting complexes in terms of the subgroup of the Weyl group and this observation is crucial in this paper. 
In section 5, we show that preprojective algebras of Dynkin type are tilting-discrete. It implies that any tilting complex is obtained by iterated mutation from an arbitrary tilting complex. 
In section 6, we show that there exists a map from the braid group to tilting complexes and  we prove that it is a bijection.

\begin{notation}
Throughout this paper, 
let $K$ be an algebraically closed field and $D:=\Hom_K(-,K)$.
For a finite dimensional algebra $\La$ over $K$, we denote by $\mod\Lambda$ the category of finitely generated right $\Lambda$-modules and by $\proj\La$ the category of finitely generated projective $\La$-modules. 
We denote by $\Db(\mod\La)$ the bounded derived category of $\mod\La$ and by $\Kb(\proj\Lambda)$ the bounded homotopy category of $\proj\La$. 
\end{notation}

\textbf{Acknowledgement.} 
The authors are deeply grateful to Osamu Iyama for his kind advice and helpful discussions. 
The second author thanks Kenichi Shimizu and Dong Yang for useful discussion.
He also thanks the Institute Mittag-Leffler and Nanjing normal university for the support and warm hospitality during the preparation of this paper.

%%%%%%%%%%%%%%%%%%%%%%%%%%%%%%%%%%%%%%%%%%%%%%%%%%%%%%%%%%%%%%%%%%%%%%%%%%%%%%%%%%%%%%%%%%%%%%%%%%%%%%%%%%%%%%%%%%%%%%%%%%%%%%%%%%%%%%%%%%%%%%%%%%%%%%

\section{Silting-discrete triangulated categories}

In this section, we study silting-discrete triangulated categories. 
In particular, we give a criterion for silting-discreteness. 
Moreover we apply this theory for tilting-discrete categories for selfinjective algebras. 
We also study a relationship between silting-discrete categories and a Bongartz-type lemma.

Throughout this section, let $\T$ be a Krull-Schmidt triangulated category and 
assume that it satisfies the following property:
\begin{itemize}
\item[$\bullet$] For any object $X$ of $\T$, 
the additive closure $\add X$ is functorially finite in $\T$. 
\end{itemize}

For example, it is satisfied if $\T$ is the homotopy category of bounded complexes of finitely generated projective modules over a finite dimensional algebra, which is a main object in this paper.  
More generally, let $R$ be a complete local Noetherian ring and $\T$ an $R$-linear idempotent-complete triangulated category such that $\Hom_\T(X,Y)$ is a finitely generated $R$-module for any object $X$ and $Y$ of $\T$. Then $\T$ satisfies the above property.

%%%%%%%%%%%%%%%%%%%%%%%%%%%%%%%%%%%%%%%%%%%%%%%%%%%%%%%%%%%%%%%%%%%%%%%%%%%%%%%%%%%%%%%%%%%%%%%%%%%%%%%%%%%%%%%%%%%%%%%%%%%%%%%%%%%%%%%%%%%%%%%%%%%%%%
\subsection{Criterions of silting-discreteness}

Let us start with recalling the definition of silting objects \cite{AI,BRT,KV}.

\begin{definition}\label{def silt obj}
\begin{itemize}
\item[(a)] We call an object $P$ in $\T$ is \emph{presilting} (respectively, \emph{pretilting})
if it satisfies $\Hom_\T(P, P[i])=0$ for any $i>0$ (respectively, $i\neq0$).
\item[(b)] We call an object $P$ in $\T$ \emph{silting} (respectively, \emph{tilting})
if it is presilting (respectively, pretilting) and the smallest thick subcategory  containing $P$ is $\T$.
\end{itemize}
We denote by $\silt\T$ (respectively, $\tilt\T$) the set of isomorphism classes of basic silting objects (respectively, tilting objects) in $\T$.

It is known that the number of non-isomorphic indecomposable summands of a silting object 
does not depend on the choice of silting objects \cite[Corollary 2.28]{AI}. 
Moreover, for objects $P$ and $Q$ of $\T$,
we write $P\geq Q$ if $\Hom_\T(P,Q[i])=0$ for any $i>0$, which gives a partial order on $\silt\T$ \cite[Theorem 2.11]{AI}.
\end{definition}

Then we give the definition of silting-discrete triangulated categories as follows. 

\begin{definition}\label{definition of silting-discrete}
\begin{itemize}
\item[(a)]
We call a triangulated category $\T$ \emph{silting-discrete}
if for any  $P\in\silt\T$ and any $\ell>0$, the set 
$$\{T\in \silt\T\ |\ P\geq T\geq P[\ell]\}$$
is finite.
Note that the property of being silting-discrete does not depend on the choice of silting objects \cite[Proposition 3.8]{A}. 
Hence 
it is equivalent to say that, for a silting object $A\in\T$ and any $\ell>0$, the set
$\{T\in \silt\T\ |\ A\geq T\geq A[\ell]\}$
is finite. 
Similarly, we call $\T$ \emph{tilting-discrete} if, for a tilting object $A\in\T$ and any $\ell>0$, the set $\{T\in \tilt\T\ |\ A\geq T\geq A[\ell] \}$ is finite. 
\item[(b)]For a silting object $P$ of $\T$,
we denote by $\tsilt_P\T$ the subset of $\silt\T$ such that $U$ with $P\geq U\geq P[1]$. 
We call $\T$ \emph{2-silting-finite} if $\tsilt_P\T$ is a finite set for any silting object $P$ of $\T$. Note that the finiteness of $\tsilt_P\T$ depends on a  silting object $P$ in general.
Similarly, we denote by $\ttilt_P\T$ the subset of $\tilt\T$ such that $U$ with $P\geq U\geq P[1]$.
\end{itemize}
\end{definition}

Moreover we recall mutation for silting objects \cite[Theorem 2.31]{AI}.

\begin{definition}\label{defsm}
Let $P$ be a basic silting object of $\T$ and decompose it as $P=X\oplus M$. 
We take a triangle
\[\xymatrix{
X \ar[r]^f & M' \ar[r] & Y \ar[r] & X[1]
}\]
with a minimal left ($\add M$)-approximation $f$ of $X$.
Then $\mu_X^+(P):=Y\oplus M$ is again a silting object, 
and we call it the \emph{left mutation} of $P$ with respect to $X$.
Dually, we define the right mutation $\mu_X^-(P)$ \footnote{
The convention of $\mu^+$ and $\mu^-$ is different from \cite{M1} in which we use the converse notation}.
Mutation will mean either left or right mutation. 
If $X$ is indecomposable, then we say that 
mutation is \emph{irreducible}. 
In this case, we have $P>\mu_X^+(P)$ and there is no silting object $Q$ satisfying $P>Q>\mu_X^+(P)$ \cite[Theorem 2.35]{AI}. 

Moreover, if $P$ and $\mu_X^+(P)$ are tilting objects, then we call it the (left) \emph{tilting mutation}. In this case, if there exists no non-trivial direct summand $X'$ of $X$ such that $\mu_{X'}^+(T)$ is tilting, then we say that tilting mutation is \emph{irreducible} (\cite[Definition 5.3]{CKL}). 
\end{definition}

We remark that all silting objects of a silting-discrete category are reachable 
by iterated irreducible mutation \cite[Corollary 3.9]{A}.

Our first aim is to show the following theorem.

\begin{theorem}\label{silting-discrete}
The following are equivalent.
\begin{itemize}
\item[(a)] $\T$ is silting-discrete.
\item[(b)] $\T$ is 2-silting-finite.
\item[(c)] For a silting object $A\in\T$, $\tsilt_P\T$ is a finite set for any silting object $P$ which is given by 
iterated irreducible left mutation from $A$.
\end{itemize}
\end{theorem}

We note that the theorem is different from \cite[Lemma 2.14]{QW}, where the partial order is defined by a finite sequence of tilts, while our partial order is valid  for any silting objects. 

Now we give some examples of silting-discrete categories. 

\begin{example}\label{silt example}
Let $\La$ be a finite dimensional algebra. Then $\Kb(\proj\La)$ is silting-discrete if 
\begin{itemize}
\item[(a)]$\La$ is a path algebra of Dynkin type, which immediately follows from the definition.  
\item[(b)]$\La$ is a local algebra \cite[Corollary 2.43]{AI}.
\item[(c)]$\La$ is a representation-finite symmetric algebra \cite[Theorem 5.6]{A}, which is also tilting-discrete.
\item[(d)]$\La$ is a derived discrete algebra of finite global dimension \cite[Proposition 6.8]{BPP}.
\item[(e)] $\Lambda$ is a Brauer graph algebra whose
Brauer graph contains at most one cycle of odd length and no
cycle of even length \cite{AAC}, which is also tilting-discrete.
\end{itemize}
\end{example}

For a proof of Theorem \ref{silting-discrete}, we will introduce the following terminology.

\begin{definition}\label{def:ms}
We define a subset of $\silt\T$   
\[\nabla_{\! A}(T):=\{U\in\silt\T\ |\ A\geq U\geq A[1]\ \mbox{and } U\geq T \},\]
where $A$ is a silting object and $T$ is a presilting object in $\T$ satisfying $A\geq T$.
Note that we have $T\geq A[\ell]$ for some $\ell\geq 0$ \cite[Proposition 2.4]{AI}. 

Moreover, we say that a silting object $P$ is \emph{minimal} in $\nabla_{\! A}(T)$ 
if it is a minimal element in the partially ordered set $\nabla_{\! A}(T)$.
\end{definition}

To keep this notation, we will make the following assumption.

\begin{assumption}
In the rest of this section, we always assume that 
 $\T$ admits a silting object $A$ and a presilting object $T$ in $\T$ satisfying $A\geq T$.
\end{assumption}

Then we give the following key proposition. 

\begin{proposition}\label{l-1}
If a silting object $P$ is minimal in $\nabla_{\! A}(T)$ and $T\geq A[\ell]$ for some $\ell>0$,
then we have $T\geq P[\ell-1]$.  
\end{proposition}

For a proof, we recall the following proposition. See \cite[Proposition 2.23, 2.24, 2.36]{AI} and \cite[Proposition 2.12]{A}.

\begin{proposition}\label{sm}
Let $P$ be a  silting object of $\T$. Then the following hold.
\begin{itemize}
\item[(a)] There exists $\ell\geq 0$ such that $P\geq T\geq P[\ell]$ if and only if there exist triangles 
\[\xymatrix@R=0.3cm{
T_1 \ar[r] & P_0 \ar[r]^(0.4){f_0} & T_0:=T \ar[r] & T_1[1], \\
&\cdots,\\
T_{\ell-1} \ar[r] & P_{\ell-2} \ar[r]^{f_{\ell-2}} & T_{\ell-2} \ar[r] & T_{\ell-1}[1], \\
T_\ell \ar[r] & P_{\ell-1} \ar[r]^{f_{\ell-1}} & T_{\ell-1} \ar[r] & P_\ell[1],\\ 
0 \ar[r] & P_\ell \ar[r]^{f_\ell} & T_\ell \ar[r] & 0,
}\]
where $f_i$ is a minimal right $(\add P)$-approximation of $T_i$ for $0\leq i\leq \ell$. 

\item[(b)] In the situation of $(a)$, if $\ell\neq 0$, then there is a non-zero direct summand $X\in\add(P_\ell)$ such that the irreducible left mutation $\mu_X^+(P)\geq T$. 
\end{itemize}
\end{proposition}

Using Proposition \ref{sm}, we give a proof of Proposition \ref{l-1}.

\begin{proof}[Proof of Proposition \ref{l-1}]
Since $P$ is minimal in $\nabla_{\! A}(T)$, we have 
$P\geq T\geq A[\ell]\geq P[\ell]$. 
Then, by Proposition \ref{sm} (a), 
there exist triangles 
\[\xymatrix@R=0.3cm{
T_1 \ar[r] & P_0 \ar[r]^(0.4){f_0} & T_0:=T \ar[r] & T_1[1], \\
&\cdots,\\
T_{\ell-1} \ar[r] & P_{\ell-2} \ar[r]^{f_{\ell-2}} & T_{\ell-2} \ar[r] & T_{\ell-1}[1], \\
T_\ell \ar[r] & P_{\ell-1} \ar[r]^{f_{\ell-1}} & T_{\ell-1} \ar[r] & P_\ell[1],\\
0 \ar[r] & P_\ell \ar[r]^{f_\ell} & T_\ell \ar[r] & 0,
}\]
where $f_i$ is a minimal right ($\add P$)-approximation of $T_i$ for $0\leq i\leq\ell$. 

Similarly, since we have $P\geq A[1]\geq P[1]$,  
there is a triangle 
\begin{equation}\label{a[1]}
\xymatrix@R=0.3cm{
Q_1 \ar[r] & Q_0 \ar[r]^{f} & A[1] \ar[r] & Q_1[1],
}\end{equation}
where $f$ is a minimal right ($\add P$)-approximation of $A[1]$ and $Q_1\in\add P$.

(i) We show that $P_\ell$ belongs to $\add Q_1$.
First, we have $\Hom_\T(T,A[1+\ell])=0$ by the definition of $T\geq A[\ell]$. 
Hence it follows from \cite[Lemma 2.25]{AI} that $\left(\add P_\ell\right)\cap\left(\add Q_0\right)=0$.

On the other hand, since $A[1]$ is a silting object, we find out that $Q_0\oplus Q_1$ is also a silting object by the sequence (\ref{a[1]}). 
From \cite[Theorem 2.18]{AI}, it is observed that $\add P=\add(Q_0\oplus Q_1)$ and hence $P_\ell$ belongs to $\add Q_1$.

(ii) We show that $T\geq P[\ell-1]$.
Suppose that $P_\ell\neq0$. Then we can take a direct summand $X\neq 0$ of $P_\ell$ such that $\mu_X^+(P)\geq T$ from Proposition \ref{sm} (b).

On the other hand, (i) implies that $X$ belongs to $\add Q_1$.
Since $P\geq A[1]\geq P[1]$, by applying Proposition \ref{sm} (b) to the sequence (\ref{a[1]}),
we see that $\mu_X^+(P)\geq A[1]$. 
Thus, one gets a silting object $\mu_X^+(P)$ such that $P>\mu_X^+(P)\geq A[1]$ satisfying $\mu_X^+(P)\geq T$, 
which is a contradiction to the minimality of $P$.
Therefore, we conclude that $P_\ell=0$. 
Hence we get $T\geq P[\ell-1]$ by Proposition \ref{sm} (a).
\end{proof}

On the other hand, we can easily check the following lemma.

\begin{lemma}\label{nonempty}
Let $A$ be a silting object. 
If $\tsilt_A\T$ is a finite set, 
then there exists a minimal element in $\nabla_{\! A}(T)$.
\end{lemma}

Then we give a proof of Theorem \ref{silting-discrete}, which provides a criterion of silting-discreteness.

\begin{proof}[Proof of Theorem \ref{silting-discrete}]
It is obvious that the implications (a)$\Rightarrow$(b)$\Rightarrow$(c) hold.

We show that the implication (c)$\Rightarrow$(a) holds.
Let $T$ be a silting object such that $A\geq T\geq A[\ell]$ for some $\ell>0$.
Since $\tsilt_A\T$ is a finite set,
there exists a minimal object $P$ in $\nabla_{\! A}(T)$. 
Hence we get $P\geq T\geq P[\ell-1]$ by Proposition \ref{sm}.

Thus, one obtains
\[\{T\in\silt\T\ |\ A\geq T\geq A[\ell] \}\subseteq \bigcup_{
P\in\tsilt_A\T
}
\{U\in\silt\T\ |\ P\geq U\geq P[\ell-1] \}.\]
By \cite[Theorem 3.5]{A}, the finiteness of $\tsilt_A\T$ implies that
$P$ can be obtained from $A$ by iterated irreducible left mutation.
Therefore, our assumption yields that $\tsilt_P\T$ is also a finite set.
Repeating this argument leads to the assertion.
\end{proof}

Moreover, using an analogous statement of Proposition \ref{sm} (see \cite[section 5]{CKL}), we give 
a criterion for tilting-discreteness for selfinjective algebras as follows.

\begin{corollary}\label{tilting-discrete}
Let $\La$ be a basic finite dimensional selfinjective algebra and $\T:=\Kb(\proj\Lambda)$.
Then the following are equivalent. 
\begin{itemize}
\item[(a)] $\T$ is tilting-discrete.
\item[(b)] $\T$ is 2-tilting-finite.
\item[(c)] $\ttilt_P\T$ is a finite set for any tilting object $P$ which is given by 
iterated irreducible left tilting mutation from $\La$.
\end{itemize}
\end{corollary}

\begin{proof}
It is obvious that the implications (a)$\Rightarrow$(b)$\Rightarrow$(c) hold.

We show that the implication (c)$\Rightarrow$(a) holds.
Let $T$ be a tilting object such that $\La\geq T\geq \La[\ell]$ for some $\ell>0$.
Since $\ttilt_\La\T$ is a finite set, 
there exists a minimal tilting object $P$ in $\nabla_{\! \La}(T)$.  
Then, by \cite[Proposition 5.10,Theorem 5.11]{CKL}, the same argument of Proposition \ref{sm} works for tilting objects and irreducible tilting mutation. Hence we obtain Proposition \ref{l-1} for tilting objects and 
one can get $P\geq T\geq P[\ell-1]$. 

Thus, one obtains
\[\{T\in\tilt\T\ |\ \La\geq T\geq \La[\ell] \}\subseteq \bigcup_{
P\in\ttilt_\La\T
}
\{U\in\tilt\T\ |\ P\geq U\geq P[\ell-1] \}.\]
By \cite[Theorem 5.11]{CKL}, the finiteness of $\ttilt_\La\T$ implies that
$P$ can be obtained from $\La$ by iterated irreducible left tilting mutation.
Therefore, our assumption yields that $\ttilt_P\T$ is also a finite set.
Repeating this argument leads to the assertion. 
\end{proof}

Finally, as an application of Theorem \ref{silting-discrete}, we show that silting-discrete categories satisfy a Bongartz-type lemma. 
For this purpose, we give the following definition.

\begin{definition}
We call a presilting object $T$ in $\T$ \emph{partial silting}
if it is a direct summand of some silting object, that is, 
there exists an object $T'$ such that $T\oplus T'$ is a silting object. 
\end{definition}

One of the important questions is if any presilting object is partial silting or not   \cite[Question 3.13]{BY}. 
We will show that it has a positive answer in the case of silting-discrete categories.

Let us recall the following result.

\begin{proposition}\cite[proposition 2.16]{A}\label{2.16}
Let $T$ a presilting object in $\T$. 
If $A\geq T\geq A[1]$, then $T$ is partial silting.
\end{proposition}

Then we can improve Proposition \ref{2.16} as follows.

\begin{proposition}\label{inductive silting}
Let $T$ a presilting object in $\T$ such that $A\geq T$. 
Assume that for any silting object $B$ in $\T$ such that $A\geq B\geq T$, there exists a minimal  object in $\nabla_{\! B}(T)$. 

Then there exists a silting object $P$ in $\T$ satisfying $P\geq T\geq P[1]$. 
In particular, $T$ is partial silting.
\end{proposition}

\begin{proof}
We can take $\ell\geq 0$ such that $A\geq T\geq A[\ell]$ by \cite[Proposition 2.4]{AI}.  
It is enough to show the statement for $\ell\geq2$.
Since there is a minimal silting object in $\nabla_{\! A}(T)$, where we denote it by $A_1$, we have $A_1\geq T\geq A_1[\ell-1]$ by Proposition \ref{l-1}.
By our assumption, we can repeat this argument
and we obtain a sequence 
\[A=A_0\geq A_1\geq\cdots\geq A_{\ell-1}\geq T\geq A_{\ell-1}[1]\geq\cdots\geq A_1[\ell-1]\geq A[\ell],\] 
where $A_{i+1}$ is a minimal object in $\nabla_{\! A_{i}}(T)$ for $0\leq i\leq \ell-2$.
Thus, we get the desired silting object $P:=A_{\ell-1}$.

The second assertion immediately follows from the first one and Proposition \ref{2.16}.
\end{proof}

As a consequence, we obtain the following theorem.

\begin{theorem}\label{Bongartz: silting-discrete}
If $\T$ is silting-discrete, then any presilting object is partial silting.
\end{theorem}

\begin{proof}
Take a presilting object $T$ in $\T$. 
If $T$ is presilting, then so is $T[i]$ for any $i$.
Hence we can assume that $A\geq T$. 
Then, by Theorem \ref{silting-discrete} and Lemma \ref{nonempty}, $\T$ satisfies the assumption of Proposition \ref{inductive silting} and hence we can obtain the conclusion.
\end{proof}

We remark that in \cite[secion 5]{BPP} the authors also discuss the Bongartz completion using a different type of partial orders.

%%%%%%%%%%%%%%%%%%%%%%%%%%%%%%%%%%%%%%%%%%%%%%%%%%%%%%%%%%%%%%%%%%%%%%%%%%%%%%%%%%%%%%%%%%%%%%%%%%%%%%%%%%%%%%%%%%%%%%%%%%%%%%%%%%%%%%%%%%%%%%%%%%%%%%%%%%%%%%%%%%%%%%%%%%%%%%%%%%%%%%%%%%%%%%%%%%%%%%%%%%%%

\section{Basic properties of preprojective algebras of Dynkin type}
In this section, we review some definitions and results we will use in the rest of this paper.

\subsection{Preprojective algebras}\label{nakayama per}

Let $Q$ be a finite connected acyclic quiver. We denote by $Q_0$ vertices of $Q$ and by $Q_1$ arrows of $Q$. 
We denote by $\overline{Q}$ the double quiver of $Q$, which is obtained by adding an arrow $a^*:j\to i$ for each arrow $a:i\to j$ in $Q_1$.
The \emph{preprojective algebra} $\La_Q=\La$ associated to $Q$ is the algebra $K\overline{Q}/I$, where $I$ is the ideal in the path algebra $K\overline{Q}$ generated by the relation of the form:
\[\sum_{a\in Q_1}(aa^*-a^*a).\]

We remark that $\La$ does not depend on the orientation of $Q$. 
Hence, for a graph $\Delta$, we define the preprojective algebra by $\La_\Delta=\La_Q$, where $Q$ is a quiver whose underlying graph is $\Delta$.  
We denote by $\Delta_0$ vertices of $\Delta$. 

Let $\Delta$ be a Dynkin graph (by Dynkin graph we always mean the one of type ADE). The preprojective algebra of $\Delta$ is finite dimensional and selfinjective \cite[Theorem 4.8]{BBK}. 
Without loss of generality, we may suppose that vertices are given as Figure \ref{figure} (This is because these choices make the  argument simple) and  
let $e_i$ be the primitive idempotent of $\La$ associated with $i\in \Delta_0$. 
We denote the Nakayama permutation of $\La$ by $\iota:\Delta_0\to \Delta_0$ (i.e. $D(\La e_{\iota(i)})\cong e_{i}\La$). 
Then, one can check that we have $\iota=\id$ if  $\Delta$ is type $\type{D}_{2n},\type{E}_7$ and $\type{E}_8$. Otherwise, we have $\iota^2=\id$ and it is given as follows.

\[\left\{\begin{array}{ll}
\iota(1)=1\ \mbox{and}\ \iota(i)=i+n-1\ \mbox{for}\ i\in\{2,\cdots, n\}\ \ \ \ \ &\mbox{if $\type{A}_{2n-1}$}\\
\iota(i)=i+n\ \mbox{for}\ i\in\{1,\cdots, n\} \ \ \ &\mbox{if $\type{A}_{2n}$}\\
\iota(1)=n\ \mbox{and}\ \iota(i)=i\ \mbox{for}\ i\notin\{1,n\} \ \ \ &\mbox{if $\type{D}_{2n+1}$}\\
\iota(3)=5, \iota(4)=6\ \mbox{and}\ \iota(i)=i \ \mbox{for}\ i\in\{1,2\}&\mbox{if $\type{E}_{6}$.}\\
\end{array}\right.\]

\begin{figure}
\[\def\arraystretch{2}
\begin{array}{ll}
\type{A}_{2n-1}\ : &
\begin{array}{c}
\xymatrix@C5pt@R10pt{
n \ar@{-}[r] & \cdots \ar@{-}[r] & 2 \ar@{-}[r] &\ar@{-}[r]  1 &\ar@{-}[r](n+1)& \cdots \ar@{-}[r] &   (2n-1). 
}
\end{array} \\
\type{A}_{2n}\ : &
\begin{array}{c}
\xymatrix@C5pt@R10pt{
n \ar@{-}[r]  & \cdots \ar@{-}[r] & 2 \ar@{-}[r] &\ar@{-}[r]  1 &\ar@{-}[r](n+1)& \cdots \ar@{-}[r]   &   2n. 
}
\end{array} \\

\type{B}_n\ (n\geq1): &
\begin{array}{c}
\xymatrix{
1 \ar@{-}[r]^{4} & 2 \ar@{-}[r]  & \cdots &\ar@{-}[r]  & n -1\ar@{-}[r]  & n. 
}
\end{array} \\
\type{D}_n\ (n\geq4): &
\begin{array}{c}
\xymatrix@R=0.3cm{
1  \ar@{-}[rd]&   &   &        & \\
  & 2 \ar@{-}[r]  &3 \ar@{-}[r]  & \cdots \ar@{-}[r]  & n-1.  \\
n \ar@{-}[ru] &   &   &        &
}
\end{array} \\
\type{E}_n\ (n=6,7,8): &
\begin{array}{c}
\xymatrix{
  & & 1  \ar@{-}[d]&   &   &        & \\
4 \ar@{-}[r]  & 3 \ar@{-}[r]  &
2 \ar@{-}[r]   & 5 \ar@{-}[r] & \cdots \ar@{-}[r]& n. 
}
\end{array} \\
\type{F}_4\ : &
\begin{array}{c}
\xymatrix{
1 \ar@{-}[r] & 2 \ar@{-}[r]^4 & 3 \ar@{-}[r]  & 4. 
}
\end{array} \\

\end{array}\]
\caption{}\label{figure}\end{figure}

\subsection{Weyl group}\label{relation}

Let $\Delta$ be a graph given as Figure \ref{figure}. 
The \emph{Weyl group} $W_\Delta$ associated to $\Delta$ is defined by the generators 
$s_i$ and relations 
$(s_is_j)^{m(i,j)}=1$, 
where  

\[
m(i,j):=\left\{\begin{array}{ll}
1\ \ \ \ \ &\mbox{if $i=j$,}\\
2\ \ \ \ \ &\mbox{if no edge between $i$ and $j$ in $\Delta$,}\\
3\ \ \ \ \ &\mbox{if there is an edge $i\stackrel{ }{\mbox{---}}j$ in $\Delta$,}\\
4\ \ \ \ &\mbox{if  there is an edge $i\stackrel{4}{\mbox{---}}j$ in $\Delta$.}\\
\end{array}\right.\]

For $w\in W_\Delta$, we denote by $\ell(w)$ the length of $w$.
 
Let $\Delta$ be a Dynkin graph, $\La$ the preprojective algebra and $\iota$ the Nakayama permutation of $\Lambda$. 
Then $\iota$ acts on an element of the Weyl group $W_\Delta$ by $\iota(w):=s_{\iota(i_1)}s_{\iota(i_2)}\cdots s_{\iota(i_k)}$ 
for $w=s_{i_1}s_{i_2}\cdots s_{i_k}\in W_\Delta$. 
We define the subgroup $W^\iota_\Delta$ of $W_\Delta$ by 
\[W^\iota_\Delta:=\{w\in W\ |\ \iota(w)=w \}.\]

Let $w_0$ be the longest element of $W_\Delta$. 
Note that we have $w_0ww_0=\iota(w)$ for $w\in W_\Delta$ (\cite{ES}). 
In particular we have $w_0w=ww_0$ for any $W^\iota_\Delta$.

Moreover we have the following result.

\begin{theorem}\label{folding}
Let $\Delta$ be a Dynkin (ADE) graph whose vertices are given as Figure \ref{figure} 
and $W_\Delta$ the Weyl group of $\Delta$. 
Let $\Delta^{\rm f}$ be a graph given by the following type. 
\[\begin{array}{|c|c|c|c|c|c|c|c|}
\hline                                                                       
\Delta & \type{A}_{2n-1},\type{A}_{2n}&\type{D}_{2n} &\type{D}_{2n+1} & \type{E}_6  & \type{E}_7 & \type{E}_8 \\ \hline
\Delta^{\rm f} &\type{B}_n  &\type{D}_{2n}        &\type{B}_{2n}      & \type{F}_4 & \type{E}_7  & \type{E}_8 \\ \hline
\end{array}\]

Then we have $W_\Delta^\iota=\langle t_i\ |\ i \in \Delta^{\rm f}_0\rangle$, where 
\[\tag{T}\label{fold}
t_i:=\left\{\begin{array}{ll}
\ s_i & \mbox{if  $i=\iota(i)$ in $\Delta$},\\
\ s_is_{\iota(i)}s_i & \mbox{if there is an edge $i\stackrel{ }{\mbox{---}}\iota(i)$ in $\Delta$},\\
\ s_is_{\iota(i)} & \mbox{if no edge between $i$ and $\iota(i)$ in $\Delta$},\\
\end{array}\right.\]

and $W_\Delta^\iota$ is isomorphic to $W_{\Delta^{\rm f}}$. 
\end{theorem}

\begin{proof}
This follows from the above property of the 
Nakayama permutation and \cite[Chapter 13]{C}.
\end{proof}

For the convenience, we introduce the following terminology.

\begin{definition}\label{def fold}
We call the graph $\Delta^{\rm f}$ given in Theorem \ref{folding} the \emph{folded graph} of $\Delta$.
\end{definition}

\begin{example}
\begin{itemize}
\item[(a)]Let $\Delta$ be a graph of type $\type{A}_5$.
Then one can check that $W_\Delta^\iota$ is given by $\langle s_1,s_2s_4,s_3s_5\rangle$ and this group 
is isomorphic to $W_{\Delta^{\rm f}}$, where $\Delta^{\rm f}$ is a graph of type $\type{B}_3$.
\item[(b)] Let $\Delta$ be a graph of type $\type{A}_6$.
Then one can check that $W_\Delta^\iota$ is given by $\langle s_1s_4s_1,s_2s_5,s_3s_6\rangle$ and this group  
is isomorphic to $W_{\Delta^{\rm f}}$, where $\Delta^{\rm f}$ is a graph of type $\type{B}_3$. 
\item[(c)]Let $\Delta$ be a graph of type $\type{D}_5$.
Then one can check that $W_\Delta^\iota$ is given by $\langle s_1s_5,s_2,s_3,s_4\rangle$ and this group is isomorphic to $W_{\Delta^{\rm f}}$, where $\Delta^{\rm f}$ is a graph of type $\type{B}_4$.
\item[(d)]Let $\Delta$ be a graph of type $\type{E}_6$.
Then one can check that $W_\Delta^\iota$ is given by $\langle s_1,s_2,s_3s_5,s_4s_6\rangle$ and this group is isomorphic to $W_{\Delta^{\rm f}}$, where $\Delta^{\rm f}$ is a graph of type $\type{F}_4$.
\end{itemize}
\end{example}

%%%%%%%%%%%%%%%%%%%%%%%%%%%%%%%%%%%%%%%%%%%%%%%

\subsection{Support $\tau$-tilting modules and two-term silting complexes.} 
In this subsection, we briefly recall the notion of support $\tau$-tilting modules introduced in \cite{AIR}, and its relationship with silting complexes. We refer to \cite{AIR,IR2} for a background of support $\tau$-tilting modules.

Let $\La$ be a finite dimensional algebra and we denote by $\tau$ the AR translation \cite{ARS}.

\begin{definition}
\begin{itemize}
\item[(a)] We call $X$ in $\mod\Lambda$ \emph{$\tau$-rigid} if $\Hom_{\Lambda}(X,\tau X)=0$.
\item[(b)] We call $X$ in $\mod\Lambda$ \emph{$\tau$-tilting} if $X$ is $\tau$-rigid and $|X|=|\Lambda|$, where $|X|$ denotes the number of non-isomorphic indecomposable direct summands of $X$.
\item[(c)] We call $X$ in $\mod\Lambda$ \emph{support $\tau$-tilting} if there exists an idempotent $e$ of $\La$ such that $X$ is a $\tau$-tilting $(\La/\langle e\rangle)$-module.
\end{itemize}
We can also describe these notions as pairs as follows.

\begin{itemize}
\item[(d)] We call a pair $(X,P)$ of $X\in\mod\La$ and $P\in\proj\Lambda$ 
 \emph{$\tau$-rigid} if $X$ is $\tau$-rigid and $\Hom_\Lambda(P,X)=0$.
\item[(e)] We call a $\tau$-rigid pair $(X,P)$ a \emph{support $\tau$-tilting}   (respectively, \emph{almost complete support $\tau$-tilting}) pair if $|X|+|P|=|\Lambda|$ (respectively, $|X|+|P|=|\Lambda|-1$). 
\end{itemize}
\end{definition}

We say that $(X,P)$ is \emph{basic} if $X$ and $P$ are basic, and we say that $(X,P)$ is a \emph{direct summand} of $(X',P')$ if $X$ is a direct summand of $X'$ and $P$ is a direct summand of $P'$. Note that 
a basic support $\tau$-tilting module $X$ determines a basic support $\tau$-tilting pair $(X,P)$ uniquely \cite[Proposition 2.3]{AIR}. Hence we can identify basic support $\tau$-tilting modules with basic support $\tau$-tilting pairs. 
We denote by $\sttilt\La$ the set of isomorphism classes of basic support $\tau$-tilting $\La$-modules. 

Finally we recall an important relationship between support $\tau$-tilting modules and two-term silting complexes. 
We write $\silt\Lambda:=\silt\Kb(\proj\La)$ and $\tilt\Lambda:=\tilt\Kb(\proj\La)$ for simplicity.
We denote by $\tsilt\Lambda$ (respectively, $\ttilt\Lambda$) the subset of $\silt\Lambda$ (respectively, $\tilt\Lambda$) consisting of two-term (i.e. it is concentrated in the degree 0 and $-1$) complexes. Note that a complex $T$ is two-term if and only if $\La\geq T\geq\La[1]$.

Then we have the following nice correspondence. 

\begin{theorem}\cite[Theorem 3.2, Corollary 3.9]{AIR}\label{tau-2silt}
Let $\Lambda$ be a finite dimensional algebra. 
There exists a bijection $\Psi:\sttilt\Lambda\longrightarrow\tsilt\Lambda,$
\[(X,P)\mapsto \Psi(X,P):=\left\{\begin{array}{cccc}
\stackrel{-1}{P_X^1}&\stackrel{f}{\longrightarrow}&\stackrel{0}{P_X^0}&\\
&\oplus&&\in\Kb(\proj\La)\\
P&&&
\end{array}\right.\]
where $\xymatrix@C10pt@R10pt{P_X^1\ar[r]^f& P_X^0\ar[r]^{}&X\ar[r]& 0}$ is a minimal projective presentation of $X$. 
Moreover, it gives an isomorphism of the partially ordered sets between $\sttilt\Lambda$ and $\tsilt\Lambda$. 
\end{theorem}

By the above correspondence, we can give a description of two-term silting complexes by calculating support $\tau$-tilting modules, which is much simpler than calculations of 
two-term silting complexes.

%%%%%%%%%%%%%%%%%%%%%%%%%%%%%%%%%%%%%%%%%%%%%%%%%%%%%%%%%%%%%%%%%%%%%%%%%%%%%%%%%%%%%%%%%%%%%%%%%%%%%%%%%%%%%%%%%%%%%%%%%%%%%%%%%%%%%%%%%%%%%%%%%%%%%%

\section{Two-term tilting complexes and Weyl groups}

In this section, we characterize 2-term tilting complexes in terms of the Weyl group. In particular, we provide a complete description of 2-term tilting complexes.

Throughout this section, let $\Delta$ be a Dynkin (ADE) graph with $\Delta_0=\{1,\ldots, n\}$, $\La$ the preprojective algebra of $\Delta$ and $I_i:=\La(1-e_i)\La$, where $e_i$ the primitive idempotent of $\La$ associated with $i\in \Delta_0$. 
We denote by $\langle I_1,\ldots,I_n\rangle$ the set of ideals of $\La$ which can be written as 
$$I_{i_1}I_{i_2}\cdots I_{i_k}$$ for some $k\geq0$ and $i_1,\ldots,i_k\in \Delta_0$. 
Note that it has recently been understood that these ideals play an important role in several situations, for example \cite{IR1,BIRS,GLS2,ORT,BK,BKT}. 

Then we use the following important results. 

\begin{theorem}\label{tau-weyl}
\begin{itemize}
\item[(a)]There exists a bijection $W_\Delta\to\langle I_1,\ldots,I_n\rangle$, which is given by $w\mapsto I_w =I_{i_1}I_{i_2}\cdots I_{i_k}$ for any reduced 
expression $w=s_{i_1}\cdots s_{i_k}$.
\item[(b)]There exist bijections between 
$$W_\Delta\ \ \ \longrightarrow\ \ \ \sttilt\La\ \ \ \longrightarrow\ \ \  \tsilt\La,$$
$$\  \ \ \ \ \ \ \ \ \ w\ \ \  \ \ \mapsto\ \ \ \ (I_w,P_w)\ \  \ \ \mapsto\ \  S_w:=\Psi(I_w,P_w).$$
\item[(c)]
The Weyl group $W_\Delta$ acts transitively and faithfully on $\tsilt\La$ by
$$s_i\cdot(S_w):=\mu_i(S_w)\cong S_{s_iw},$$ 
where $\mu_i$ is the silting mutation associated with $i\in \Delta_0$.
\end{itemize}
\end{theorem}

\begin{proof}
(a) This follows from \cite[Theorem 2.14]{M1} (\cite[III.1.9]{BIRS}).

(b) This follows from \cite[Theorem 2.21]{M1} and Theorem \ref{tau-2silt}.

(c) By \cite[Theorem 2.16]{M1}, $W_\Delta$ acts transitively and faithfully on $\sttilt\La$ by mutation of support $\tau$-tilting pairs (see \cite[Theorem 2.18, 2.28]{AIR} for mutation of support $\tau$-tilting pairs). 
On the other hand, \cite[Corollary 3.9]{AIR} implies that the bijection (b) gives the compatibility of mutation of 
support $\tau$-tilting pairs and two-term silting complexes. Hence we get the conclusion.
\end{proof}

Then, the aim of this section is to show the following result.

\begin{theorem}\label{action}
Let $\Delta$ be a Dynkin graph, $\La$ the preprojective algebra of $\Delta$ and $\iota$ the Nakayama permutation of $\La$.
\begin{itemize}
\item[(a)]
Let $\nu$ the Nakayama functor of $\Lambda$. 
Then $\nu(I_w)\cong I_w$ if and only if $\iota(w)=w.$
\item[(b)]We have a bijection 
$$W^\iota_\Delta\longrightarrow\ttilt\Lambda,\ w\mapsto S_w.$$
\item[(c)]Let $\Delta^{\rm f}$ be the folded graph of $\Delta$ (Definition \ref{def fold}) and define $\langle t_i\ |\ i\in \Delta^{\rm f}_0\rangle$ by (\ref{fold}) of Theorem \ref{folding}.  
Then 
$\langle t_i\ |\ i\in \Delta^{\rm f}_0\rangle$ acts transitively and faithfully on $\ttilt\La$. 
\end{itemize}
\end{theorem}

For a proof, we recall the notion of $g$-vectors of support $\tau$-tilting modules. 
See \cite[section 3]{M1} and \cite[section 5]{AIR} for details.

Let $K_0(\proj\La)$ be the Grothendieck group of the additive category $\proj\La$, which is isomorphic to the free abelian group $\mathbb{Z}^n$, and we identify the set of isomorphism classes of projective $\La$-modules with 
the canonical basis ${\bf e}_1,\ldots,{\bf e}_n$ of $\mathbb{Z}^n$. 

For a $\La$-module $X$, take a minimal projective presentation
\[\xymatrix{P^1_X\ar[r]& P^0_X\ar[r]&X \ar[r]&0}\]
and let 
$g(X)=(g_1(X),\cdots,g_n(X))^t:=[P^0_X]-[P^1_X]\in\mathbb{Z}^n$.
Then, for any $w\in W_\Delta$ and $i\in \Delta_0$, we define a $g$-\emph{vector} by  
$$\mathbb{Z}^n\ni g^i(w)=\left\{\begin{array}{cl} g(e_iI_w) & {\rm if}\ e_iI_w\neq 0\\ 
-{\bf e}_{\iota(i)}&{\rm if}\ e_iI_w= 0.\end{array}\right.
$$

Then we define a \emph{$g$-matrix} of a support $\tau$-tilting $\La$-module $I_w$ by 
$$g(w):=(g^1(w),\cdots,g^n(w))\in{GL}_n(\mathbb{Z}).$$ 
Note that the $g$-vectors form a basis of $\mathbb{Z}^n$ \cite[Theorem 5.1]{AIR}.

On the other hand, we define a matrix $M_\iota:=({\bf e}_{\iota(1)},\ldots,{\bf e}_{\iota(n)})\in{GL}_{n}(\mathbb{Z})$ and, for $X\in{GL}_{n}(\mathbb{Z}),$ 
we define $$\iota(X):=M_\iota\cdot X\cdot M_\iota.$$
Clearly the left multiplication (respectively, right multiplication) of $M_\iota$ to $X$ gives a permutation of $X$ from $j$-th to $\iota(j)$-th rows (respectively, columns) for any $j\in \Delta_0$ and $M_\iota^2=\id$. 
 
Moreover, we recall the following definition (cf. \cite[Definition 3.5]{M1}).

\begin{definition}\label{geo rep}\cite{BB}
The \emph{contragradient $r: W_\Delta\rightarrow {GL}_n(\mathbb{Z})$ of the geometric representation}  is defined by 
$$r(s_i)({\bf e}_j)=r_i({\bf e}_j)=\left\{\begin{array}{cl} {\bf e}_j & i\neq j\\ -{\bf e}_i
+ {\displaystyle\sum_{\begin{smallmatrix}k-i\end{smallmatrix}}} {\bf e}_{k}& i=j,\end{array}\right.
$$
where the sum is taken over all edges of $i$ in $\Delta$.
We regard $r_i$ as a matrix of ${GL}_n(\mathbb{Z})$ and this extends to a group homomorphism.
\end{definition}

Then we start with the following observation.

\begin{lemma}\label{reflectionsigma}
For any $i\in \Delta_0$, we have 
$$\iota(r_i)=r_{\iota(i)}.$$
\end{lemma}

\begin{proof}
Since the left multiplication (respectively, right multiplication) of $M_\iota$ gives a permutation of rows (respectively, columns) from $j$-th to $\iota(j)$-th for any $j\in \Delta_0$, 
this follows from the definition of $r_{i}$ and $r_{\iota(i)}$.
\end{proof}

\begin{lemma}\label{gsigma}
For any $w\in W_\Delta$, we have 
$$\iota(g(w))=g({\iota(w)}).$$
\end{lemma}

\begin{proof}
Let $w=s_{i_1}\ldots s_{i_k}$ be an expression of $w$. 
Then, by \cite[Proposition 3.6]{M1}, we conclude 
$$g(w)=r_{i_k}\ldots r_{i_1}.$$

Hence we have 
\begin{eqnarray*}
\iota(g(w))&=&M_\iota (r_{i_k}\ldots r_{i_1})M_\iota\\
&=&(M_\iota r_{i_k}M_\iota)\cdots (M_\iota r_{i_1}M_\iota) \ \ \ \ \ \   \ \ \ \ \ (M_\iota^2=\id)\\
&=&r_{\iota(i_k)}\ldots r_{\iota(i_1)}\ \ \ \ \ \ \ \ \ \ \ \ \ \  \ \ \ \ \ \ \ \ \ (\mbox{Lemma \ref{reflectionsigma}})\\
&=&g({\iota(w)}).
\end{eqnarray*}
Thus the assertion follows.
\end{proof}

Moreover, we give the following lemma. 

\begin{lemma}\label{gnu}
Let $w\in W_\Delta$. 
\begin{itemize}
\item[(a)]$\nu(I_w)$ is also a support $\tau$-tilting $\La$-module. 
In particular, there exists some $w'\in W_\Delta$ such that $\nu(I_w)\cong I_{w'}$.
\item[(b)]For the above $w'$, we have 
$$g({w'})=\iota(g(w)).$$
\end{itemize}
\end{lemma}

\begin{proof}
\begin{itemize}
\item[(a)]
Let $(I_w,P_w)$ be a basic support $\tau$-tilting pair of $\La$, where $P_w$ is the corresponding projective $\La$-module. 
By Theorem \ref{tau-2silt}, we have the two-term silting complex in $\Kb(\proj\La)$ by 
$S_w:=(P_{I_w}^1\overset{f}{\to}P_{I_w}^0)\oplus P_w[1]\in\Kb(\proj\La),$ 
where $\xymatrix@C10pt@R10pt{{P_{I_w}^1}\ar[r]^f& {P_{I_w}^0}\ar[r]^{}&I_w\ar[r]& 0}$ is a minimal projective presentation of $I_w$. 

Then $\nu(S_w)=(\nu(P_{I_w}^1){\to}\nu(P_{I_w}^0))\oplus \nu(P_w)[1]\in\Kb(\proj\La)$ is clearly a two-term silting complex. 
Hence, by Theorem \ref{tau-2silt}, 
 $(\nu(I_w),\nu(P_w))$ is also a basic support $\tau$-tilting pair of $\La$. Thus, by Theorem \ref{tau-weyl}, there exists $w'\in W_\Delta$ such that $\nu(I_w)\cong I_{w'}$.

\item[(b)]
Take $i\in \Delta_0$. First assume that $e_iI_w\neq0$ and 
take a minimal projective presentation of $e_iI_w$ 
$$P^1\to P^0\to e_iI_w\to 0.$$
By applying $\nu$ to this sequence, we have 
$$\nu(P^1)\to \nu(P^0)\to \nu(e_iI_w)\to 0.$$

Because $[\nu(e_j\La)]=[e_{\iota(j)}\La]=M_\iota[e_j\La]$ for any $j\in \Delta_0$, 
we have $[(\nu(P^0)]-[(\nu(P^1)]=M_\iota([P^0]-[P^1])
=M_\iota(g^i(w))$. 
Then, since we have $\nu(e_iI_w)\cong e_{\iota(i)}I_{w'}$, 
we obtain $g^{\iota(i)}(w')= M_\iota(g^i(w)).$ 

Next assume that $e_iI_w=0$. 
Then we have $g^i(w)=-{\bf e}_{\iota(i)}$ by the definition.
Because $\nu(e_j\La)\cong e_{\iota(j)}\La$ for any $j \in \Delta_0$, we obtain  
$g^{\iota(i)}(w')=-{\bf e}_i=M_\iota(g^i(w)).$

Consequently, we have 
\begin{eqnarray*}
g({w'})&=&(g^1(w'),\cdots,g^n(w'))\\
&=&(g^{\iota(1)}(w'),\cdots,g^{\iota(n)}(w'))\cdot M_\iota\\
&=&( M_\iota(g^{1}(w)),\cdots,M_\iota(g^{n}(w)))\cdot M_\iota\\
&=& M_\iota\cdot( g^{1}(w),\cdots, g^{n}(w))\cdot M_\iota\\
&=&\iota(g(w)).
\end{eqnarray*}
\end{itemize}
This finishes the proof.
\end{proof}

Now we recall the following nice property. 

\begin{theorem}\label{injection}\cite[Theorem 5.5]{AIR}
The map $X\to g(X)$ induces an injection from the set of isomorphism classes
of $\tau$-rigid pairs for $\La$ to $K_0(\proj\La)$.
\end{theorem}

Then we give a proof of Theorem \ref{action} as follows.

\begin{proof}[Proof of Theorem \ref{action}]
\begin{itemize}
\item[(a)]
We have the following equivalent conditions
\begin{eqnarray*}
\nu(I_w)\cong I_w&\Leftrightarrow&\iota(g(w))=g(w)\ \ \ \ \ \ \ \ \ \ \ (\mbox{Lemma }\ref{gnu}\ \mbox{and}\ \mbox{Theorem }\ref{injection})\\
&\Leftrightarrow&g({\iota(w)})=g(w)\ \ \ \ \ \ \ \ \ \ \ \ \ \ \ \ \ \ \ \ \ \ \ \ \ \ \ \ \ \ \ \ \ \ (\mbox{Lemma }\ref{gsigma})\\
&\Leftrightarrow&I_{\iota(w)}\cong I_w\ \ \ \ \ \ \ \ \ \ \ \ \  \ \ \ \ \ \ \ \ \ \ \ \ \ \ \ \ \ \ \ \ \ \ \ \ \  \ (\mbox{Theorem }\ref{injection})\\
&\Leftrightarrow&{\iota(w)}=w.\ \ \ \ \ \ \ \ \ \ \ \ \ \ \ \ \ \ \ \ \ \ \ \ \ \ \ \  \ \ \ \ \ \ \ \ \ \ \ (\mbox{Theorem }\ref{tau-weyl})
\end{eqnarray*}
Thus we get the desired result.

\item[(b)]A silting complex $S_w$ is a tilting complex if and only if 
$\nu(S_w)\cong S_w$ (see \cite[Appendix]{A}). 
Hence (a) implies that it is equivalent to say that $\iota(w)=w$. This proves our claim. 
\item[(c)] By (b) and Theorem \ref{folding}, 
the action of Theorem \ref{tau-weyl} induces the action of $\langle t_i\ |\ i\in \Delta^{\rm f}_0\rangle$ on $\ttilt\La$.
\end{itemize}
\end{proof}

\begin{example}
Let $\Delta$ be a graph of type $\type{A}_3$ and $\La$ the preprojective algebra of $\Delta$.  
Then the \emph{support $\tau$-tilting quiver} of $\La$ (\cite[Definition 2.29]{AIR}) is given as follows. 

$$
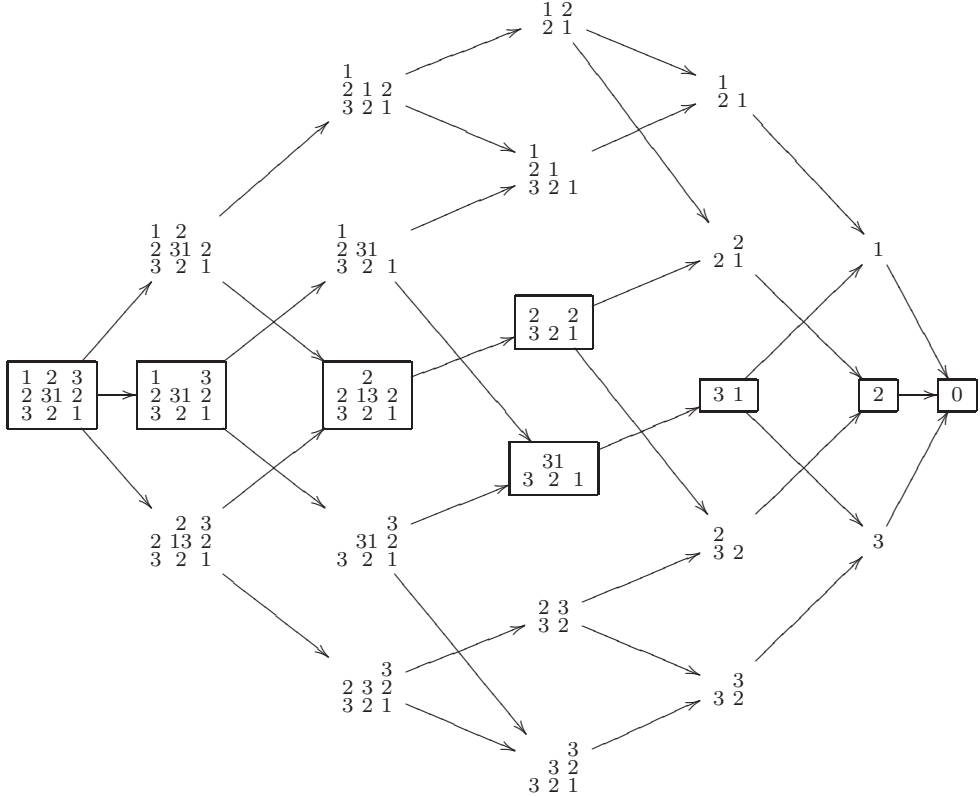
\xymatrix@C15pt@R5pt{
&   &   &   &   &    {\begin{smallmatrix} && \\ &1 &2 \\&2 &1  \end{smallmatrix}}  \ar[rrd]\ar[rrddd]   &    &   &   &   &     \\ 
&   &   & {\begin{smallmatrix}1 && \\2 &1 &2 \\3&2 &1  \end{smallmatrix}}\ar[rru]\ar[rrd]   &   &         &    & {\begin{smallmatrix} && \\ &1 & \\&2 &1  \end{smallmatrix}}\ar[rrdd]   &   &   &     \\ 
&   &   &   &   &    {\begin{smallmatrix} 1&& \\2 &1 & \\3&2 & 1 \end{smallmatrix}} \ar[rru]     &    &   &   &   &     \\ 
&  {\begin{smallmatrix}1 &2& \\ 2&31 &2 \\3&2 &1  \end{smallmatrix}} \ar[rruu]\ar[rrdd] &   & {\begin{smallmatrix} 1&& \\2 &31 & \\3&2 &1  \end{smallmatrix}}\ar[rru]\ar[rrddd]   &   &         &    & {\begin{smallmatrix} & \\  &2 \\2 &1  \end{smallmatrix}}\ar[rrdd]   &   & {\begin{smallmatrix} 1  \end{smallmatrix}}\ar[rdd]   &     \\ 
&   &   &   &   & *+[F]{  {\begin{smallmatrix} && \\2 & &2 \\3& 2&1  \end{smallmatrix}}}\ar[rru]\ar[rrddd]       &    &   &   &   &     \\ 
*+[F]{\begin{smallmatrix} 1&2&3 \\ 2&31 &2 \\3&2 &1  \end{smallmatrix}}\ar[uur]\ar[rdd]\ar[r] & *+[F] {\begin{smallmatrix} 1&&3 \\ 2&31 & 2\\3&2 &1  \end{smallmatrix}}\ar[rruu]\ar[rrdd]  &   &*+[F]{\begin{smallmatrix} &2& \\ 2&13 &2 \\3&2 &1  \end{smallmatrix}}\ar[rru] &   &         &    &*+[F] {\begin{smallmatrix} 3 &1  \end{smallmatrix}}\ar[rruu]\ar[rrdd]   &   & *+[F]{\begin{smallmatrix}2  \end{smallmatrix}}\ar[r]   &*+[F] {\begin{smallmatrix} 0\end{smallmatrix}}     \\ 
&   &   &   &   & *+[F]   {\begin{smallmatrix} && \\ & 31& \\3&2 &1  \end{smallmatrix}}\ar[rru]      &    &   &   &   &     \\ 
&{\begin{smallmatrix} &2& 3\\ 2&13 &2 \\3&2 &1  \end{smallmatrix}}\ar[rruu]\ar[rrdd] &   & {\begin{smallmatrix} && 3\\ &31 &2 \\3&2 &1  \end{smallmatrix}}\ar[rru]\ar[rrddd]   &   &         &    & {\begin{smallmatrix} & \\ 2&  \\3&2  \end{smallmatrix}}\ar[rruu]   &   & {\begin{smallmatrix} 3 \end{smallmatrix}}\ar[ruu]   &     \\ 
&   &   &   &   &{\begin{smallmatrix} & \\ 2&3  \\3&2   \end{smallmatrix}} \ar[rru]\ar[rrd]     &    &   &   &   &     \\ 
&   &   & {\begin{smallmatrix} &&3 \\ 2& 3&2 \\3&2 &1  \end{smallmatrix}} \ar[rru]\ar[rrd] &   &         &    &{\begin{smallmatrix} & \\ & 3 \\3&2   \end{smallmatrix}}\ar[rruu]   &   &   &   \\
&   &   &   &   &   {\begin{smallmatrix} &&3 \\ &3 & 2\\3&2 &1  \end{smallmatrix}}\ar[rru]      &    &   &   &   &  
  }$$
  
The framed modules indicate \emph{$\nu$-stable} modules \cite{M2} (i.e. $I_w\cong \nu(I_w)$), which is equivalent to say that $\iota(w)=w$. 
Hence Theorems \ref{folding} and \ref{action} imply that 
these modules are in bijection to the elements of the subgroup $W_\Delta^\iota=\langle s_1s_3, s_2\rangle$ and this group is isomorphic to the Weyl group of type $\type{B}_2$. 
\end{example}

%%%%%%%%%%%%%%%%%%%%%%%%%%%%%%%%%%%%%%%%%%%%%%%%%%%%%%%%%%%%%%%%%%%%%%%%%%%

\section{Preprojective algebras are tilting-discrete}
In this section, we show that preprojective algebras of Dynkin type are tilting-discrete.
It implies that all tilting complexes are connected to each other by successive tilting mutation (\cite[Theorem 5.14]{CKL}, \cite[Theorem 3.5]{A}). From this result, we can determine the derived equivalence class of the algebra. 

Throughout this section, let $\Delta$ be a Dynkin graph with $\Delta_0=\{1,\ldots, n\}$, $\La$ the preprojective algebra of $\Delta$, $e_i$ the primitive idempotent of $\La$ associated with $i\in \Delta_0$ and $\Delta^{\rm f}$ the folded graph of $\Delta$. We also keep the notation of previous sections.

The aim of this section is to show the following theorem.

\begin{theorem}\label{tilt discrete}
Let $\La$ be a preprojective algebra of Dynkin type. 
\begin{itemize}
\item[(a)] $\Kb(\proj\La)$ is tilting-discrete. 
\item[(b)] Any basic tilting complex $T$ of $\La$ satisfies $\End_{\Kb(\proj\La)}(T)\cong\La$. In particular, the derived equivalence class coincides with the Morita equivalence class.
\end{itemize}
\end{theorem}

First we introduce the following notation.

\begin{notation}
Let $\widetilde{\Delta}$ be an extended Dynkin graph obtained from $\Delta$ by adding a vertex $0$ (i.e. $\widetilde{\Delta}_0=\{0\}\cup \Delta_0$) with the associated arrows. 
Since $W_\Delta=\langle s_1,\ldots,s_n \rangle\subset W_{\widetilde{\Delta}}=\langle s_1,\ldots,s_n,s_0 \rangle$, 
we can regard elements of $W_\Delta$ as those of $W_{\widetilde{\Delta}}$. 
We denote by $\wLa$ the $\mathfrak{m}$-adic completion of the preprojective algebra of $\widetilde{\Delta}$, where $\mathfrak{m}$ is the ideal generated by all arrows. It implies that 
the Krull-Schmidt theorem holds for finitely generated projective $\wLa$-modules.
Moreover we denote by $\I_i:=\wLa(1-e_i)\wLa$, where $e_i$ is the primitive idempotent of $\wLa$ associated with $i\in \widetilde{\Delta}_0$. 

Recall that, by Theorem \ref{tau-weyl}, we have a bijection between 
$W_{\widetilde{\Delta}}$ and $\langle \I_1,\ldots,\I_n,\I_0 \rangle$ \cite[III.1.9]{BIRS} and hence 
for each element $w\in W_{\widetilde{\Delta}}$, we can define $\I_w:=\I_{i_1}\cdots \I_{i_k},$ where $w=s_{i_1}\cdots s_{i_k}$ is a reduced expression. 
Furthermore, it is shown that $\I_{w}$ is a tilting $\wLa$-module \cite[Theorem III.1.6]{BIRS}.

Note that if $i\neq0\in\widetilde{\Delta}_0$, then we have 
$$\Lambda=\widetilde{\Lambda}/{ \langle e_0\rangle}\ \textnormal{and} \ I_i={\I_i}/{ \langle e_0\rangle}.$$
In particular, for $w\in W_{\Delta}$, we have $\I_w/\langle e_0\rangle=I_w$ and hence $\wLa/\I_w\cong\La/I_w$.

\end{notation}

Recall that we can describe the two-term silting complex of $\Kb(\proj\La)$ by $$S_w:=\left\{\begin{array}{cccc}
\stackrel{}{P_{I_w}^1}&\stackrel{f}{\longrightarrow}&\stackrel{}{P_{I_w}^0}\\
&\oplus&&\\
P_w&&
\end{array}\right.$$
where $\xymatrix@C10pt@R10pt{{P_{I_w}^1}\ar[r]^f& {P_{I_w}^0}\ar[r]^{}&I_w\ar[r]& 0}$ is a minimal projective presentation of $I_w$.

Then we show that $\I_{w}\Lwotimes\La$ gives a two-term silting complex $S_w$.

\begin{proposition}\label{left str}
For $w\in W_\Delta$, $\I_{w}\Lwotimes\La$ is isomorphic to $S_w$ in $\Db(\mod\La)$. 
\end{proposition}

\begin{proof}
Since $\I_{w}$ is a tilting $\wLa$-module, we have a minimal projective resolution as follows
\begin{equation*}
\xymatrix@C20pt@R20pt{ 0\ar[r]& \widetilde{P}_1\ar[r]^{g }  & \widetilde{P}_0   \ar[r]^{} &\I_{w} \ar[r]&0.  }
\end{equation*} 
By applying the functor $-\otimes_{\wLa}\La$, we have the following exact sequence \cite[Proposition 3.2]{M1}
$$0\to\nu^{-1}(\La/I_{w})\to{\widetilde{P}_1\otimes_{\wLa}\La}\overset{g\otimes\La}{\to}{\widetilde{P}_0\otimes_{\wLa}\La}\to\I_{w}\otimes_{\wLa}\La\to0.$$ 

Because we have an isomorphism in $\Db(\mod{\wLa})$ 
$$\I_{w}\Lwotimes\La\cong(\cdots\to0\to{\widetilde{P}_1\otimes_{\wLa}\La}\overset{g\otimes\La}{\to}{\widetilde{P}_0\otimes_{\wLa}\La}\to0\to\cdots),$$ 
one can check that $\I_{w}\Lwotimes\La$ is isomorphic to $S_w$ (Theorem \ref{tau-2silt}).
\end{proof}

For $w\in W_\Delta$, we denote the inclusion by $\inc:\I_{w}\hookrightarrow \wLa$.
Then we show the following lemma.

\begin{lemma}\label{longest}
Let $w_0$ be the longest element of $W_\Delta$. 
For $w\in W_{\Delta}^\iota$, we have isomorphisms 
$p:\I_w\Lwotimes\I_{w_0}\to\I_{w_0}\Lwotimes\I_w$ and $q:\I_w\Lwotimes\wLa\to\wLa\Lwotimes\I_w$, which make the following  diagram commutative

$$\xymatrix@C25pt@R15pt{\I_{w}\Lwotimes\I_{w_0}\ar[r]^{\id\otimes\inc}\ar[d]_{\cong}^p&\I_{w}\Lwotimes\wLa\ar[d]_{\cong}^{q}\\
\I_{w_0}\Lwotimes\I_{w}\ar[r]^{\inc\otimes\id}&\wLa\Lwotimes\I_{w}.}$$

\end{lemma}

\begin{proof}
Because $\ell(w_0w^{-1})+\ell(w)=\ell(w_0)$, \cite[Propositions II.1.5(a), II.1.10.]{BIRS} gives the 
following commutative diagram  

$$\xymatrix@C35pt@R15pt{\I_{w_0}\ar[r]^{\inc}&\wLa\ar[d]_{\cong}\\
\I_{w_0w^{-1}}\Lwotimes\I_{w}\ar[r]^{\inc\otimes\inc}\ar[u]^{\cong}&\wLa\Lwotimes\wLa,}$$  
and hence we have

$$\xymatrix@C35pt@R15pt{\I_{w}\Lwotimes\I_{w_0}\ar[r]^{\id\otimes\inc}&\I_{w}\Lwotimes\wLa\ar[r]^{\inc\otimes\id}&\wLa\Lwotimes\wLa\ar[d]_{\cong}\\
\I_{w}\Lwotimes\I_{w_0w^{-1}}\Lwotimes\I_{w}\ar[rr]^{\inc\otimes\inc\otimes\inc}\ar[u]^{\cong}&&\wLa\Lwotimes\wLa\Lwotimes\wLa.}$$

Since $w\in W_\Delta^\iota$, we have ${w}_0{w}={w}{w}_0$ (subsection \ref{relation}) and hence $\I_{w_0w^{-1}}=\I_{w^{-1}w_0}$. 
Then similarly we have the following commutative diagram

$$\xymatrix@C35pt@R15pt{\I_{w}\Lwotimes\I_{w^{-1}w_0}\Lwotimes\I_{w}\ar[rr]^{\inc\otimes\inc\otimes\inc}\ar[d]_{\cong}&&\wLa\Lwotimes\wLa\Lwotimes\wLa\\
\I_{w_0}\Lwotimes\I_{w}\ar[r]^{\inc\otimes\id}&\wLa\Lwotimes\I_{w}\ar[r]^{\id\otimes\inc}&\wLa\Lwotimes\wLa\ar[u]^{\cong}.}$$ 

Moreover we have the following commutative diagram

$$\xymatrix@C35pt@R15pt{\I_{w}\Lwotimes\wLa\ar[d]_{\cong}\ar[r]^{\inc\otimes\id}&\wLa\Lwotimes\wLa\ar[d]_{\cong}\\
\wLa\Lwotimes\I_w\Lwotimes\wLa\ar[r]^{\id\otimes\inc\otimes\id}&\wLa\Lwotimes\wLa\Lwotimes\wLa\\
\ar[u]^{\cong}\wLa\Lwotimes\I_w\ar[r]^{\id\otimes\inc}&\wLa\Lwotimes\wLa\ar[u]^{\cong}.}$$

Put $L:=\I_{w}\Lwotimes\I_{w^{-1}w_0}\Lwotimes\I_{w}$. Consider a morphism $u:L\to\I_w$ and 
the triangule $\xymatrix@C10pt@R10pt{
\cdots\ar[r]&\wLa/\I_w[-1]\ar[r]&\I_{w}\ar[r]^{\inc}& \wLa \ar[r]& \wLa/\I_w\ar[r]&\dots}.$
If $\inc\circ u=0$, then there exsits a map $v:L\to\wLa/\I_w[-1]$ which makes the diagram commutative. 
$$\xymatrix@C10pt@R20pt{&&L\ar[d]^{u}\ar@{-->}[ld]_{v}&&&\\
\cdots\ar[r]&\wLa/\I_w[-1]\ar[r]&\I_{w}\ar[r]^{\inc}& \wLa \ar[r]& \wLa/\I_w\ar[r]&\dots\ .}$$

Because $H^i(L)=0$ for any $i>0$, we get $v=0$ and hence $u=0$.
Thus the above diagrams provide required morphisms.
\end{proof}

From the above results, we have the following nice consequence.

\begin{proposition}\label{endo silt}
For any $w\in W_\Delta^\iota$, we have an isomorphism 
 $$\End_{\Kb(\proj\La)}(\I_{w}\Lwotimes\La)\cong\La.$$
In particular, the endomorphism algebra of any basic two-term tilting complex is isomorphic to $\La$.
\end{proposition}

\begin{proof}
Let $w_0$ be the longest element of $W_\Delta$. 
Since $\I_{w_0}=\langle e_0\rangle$, we have the following exact sequence
$$\xymatrix@C30pt@R30pt{0\ar[r]&\I_{w_0}\ar[r]^{}& \wLa \ar[r]& \La\ar[r]&0.}$$ 

Then applying the functor $\I_{w}\Lwotimes-$ to the exact sequence, we have the triangle 
$$\xymatrix@C30pt@R15pt{\I_{w}\Lwotimes\I_{w_0}\ar[r]^{}&\I_{w}\Lwotimes\La\ar[r]^{ }&\I_{w}\Lwotimes\wLa\ar[r]&\I_{w}\Lwotimes\I_{w_0}[1].}$$ 

Similarly, applying the functor $-\Lwotimes\I_{w}$ to the first exact sequence, we have the triangle 

$$\xymatrix@C30pt@R15pt{\I_{w_0}\Lwotimes\I_{w}\ar[r]^{}&\wLa\Lwotimes\I_{w}\ar[r]^{ }&\La\Lwotimes\I_{w}\ar[r]&\I_{w_0}\Lwotimes\I_{w}[1].}$$ 

By Lemma \ref{longest}, we have the 
following commutative diagram

$$\xymatrix@C25pt@R15pt{\I_{w}\Lwotimes\I_{w_0}\ar[r]^{}\ar[d]^p_{\cong}&\I_{w}\Lwotimes\wLa\ar[r]^{ }\ar[d]^q_{\cong}&\I_{w}\Lwotimes\La\ar[r]\ar@{.>}[d]^r&\I_{w}\Lwotimes\I_{w_0}[1]\ar[d]\\
\I_{w_0}\Lwotimes\I_{w}\ar[r]^{}&\wLa\Lwotimes\I_{w}\ar[r]^{ }&\La\Lwotimes\I_{w}\ar[r]&\I_{w_0}\Lwotimes\I_{w}[1],}$$ 
and the isomorphism $r$. 
Because $\I_w$ is a tilting module \cite[Theorem III.1.6]{BIRS} and we have $\wLa\cong\Hom_{\wLa}(\I_{w},\I_{w})$ \cite[Proposition II.1.4]{BIRS}, we obtain  
\begin{eqnarray*}
\RHom_{\La}(\I_{w}\Lwotimes\La,\I_{w}\Lwotimes\La)&\cong&\RHom_{\wLa}(\I_{w},\I_{w}\Lwotimes\La)\\ 
&\cong&\RHom_{\wLa}(\I_{w},\La\Lwotimes\I_{w})\\
&\cong&\La\Lwotimes\RHom_{\wLa}(\I_{w},\I_{w})\\
&\cong&\La\Lwotimes\wLa\\ 
&\cong&\La.
\end{eqnarray*}
Then by taking 0-th part, we get the assertion. 
The second statement immediately follows from the first one, Lemma \ref{action} and Proposition \ref{left str}.
\end{proof}

\begin{corollary}\label{cor end iso}
Let $T$ be a tilting complex which is given by iterated irreducible left tilting mutation from $\La$. 
Then we have $$\End_{\Kb(\proj\La)}(T)\cong\La.$$
\end{corollary}

\begin{proof}
Let $T=\bmu^+_{(\ell)}\circ\cdots\circ\bmu^+_{(1)}(\La)$, where $\bmu$ denotes by irreducible left tilting mutation. 
We proceed by induction on $\ell$. 
Assume that, for $T'=\bmu^+_{(\ell-1)}\circ\cdots\circ\bmu^+_{(1)}(\La)$, 
we have $\End_{\Kb(\proj\La)}(T')\cong\La$. 
Then we have an equivalence $F:\Kb(\proj\La)\to\Kb(\proj\La)$ such that $F(T')\cong\La$ \cite{Ric}. 
Therefore we have $\End_{\Kb(\proj\La)}(\bmu^+_{(\ell)}(T'))
\cong\End_{\Kb(\proj\La)}(\bmu^+_{(\ell)}(\La))$ and hence it is isomorphic to $\La$ by Proposition \ref{endo silt}. 
\end{proof}

Now we give a proof of Theorem \ref{tilt discrete}.

\begin{proof}[Proof of Theorem \ref{tilt discrete}] 
(a) We will check the condition (c) of Corollary \ref{tilting-discrete}. 

Recall that  $\ttilt_T\La:=\{U\in\tilt\La\ |\ T\geq U\geq T[1]\}$.  We denote by $\sharp\ttilt_T\La$ the number of $\ttilt_T\La$. 

By Theorem \ref{action}, the set $\ttilt_{\La}\La=\ttilt\La$ is finite. 
Let $T$ be a tilting complex which is given by iterated irreducible left tilting mutation from $\La$. 
Then we have $\End_{\Kb(\proj\La)}(T)\cong\La$ from Corollary \ref{cor end iso}. 
Therefore, we have an equivalence $F:\Kb(\proj\La)\to\Kb(\proj\La)$ such that $F(T)\cong\La$ and hence   
we get $\sharp\{U\in\tilt\La\ |\ T\geq U\geq T[1]\}= 
\sharp\{F(U)\in\tilt\La\ |\ \La\geq F(U)\geq \La[1]\}$. Thus it is also finite and we obtain the  statement. 

(b) Let $T$ be a basic tilting complex such that $\La\geq T$. 
Since $\La$ is tilting-discrete, $T$ is obtained by iterated irreducible left tilting mutation from $\La$ \cite[Theorem 5.14]{CKL} (\cite[Theorem 3.5]{A}). 
Thus the statement follows from Corollary \ref{cor end iso}. 
Because for any tilting complex $T$, we have $\La\geq T[\ell]$ for some $\ell$ \cite[Proposition 2.4]{AI} and $\End_{\Kb(\proj\La)}(T)\cong\End_{\Kb(\proj\La)}(T[\ell])$, we get the conclusion from the above argument. 
\end{proof}

%%%%%%%%%%%%%%%%%%%%%%%%%%%%%%%%%%%%%%%%%%%%%%%%%%%%%%%%%%%%%%%%%%%%%%%%%%%%%%%%%%

\section{Tilting complexes and braid groups}\label{end alge}

In this section, we show that irreducible mutation satisfy the braid relations and we give a bijective map from the elements of the braid group and the set of tilting complexes. 

We keep the notation of previous sections.

Define $W_\Delta^\iota=\langle t_i\ |\ i\in \Delta^{\rm f}_0\rangle$ as (\ref{fold}) of Theorem \ref{folding}. 
By Theorems \ref{tau-weyl} and \ref{action}, we have $S_{t_i}=\bmu_i^+(\La)$ ($i\in \Delta_0^{\rm f}$) in $\Db(\mod \La)$, where $\bmu_i^+$ is given as a composition of left silting mutation as follows 
\[\bmu_i^+:=\left\{\begin{array}{ll}
\ \mu_{i}^+ & \mbox{if  $i=\iota(i)$ in $\Delta$},\\
\ \mu_{i}^+\circ\mu_{{\iota(i)}}^+\circ\mu_{i}^+& \mbox{if there is an edge $i\stackrel{ }{\mbox{--}}\iota(i)$ in $\Delta$},\\
\ \mu_{i}^+\circ\mu_{{\iota(i)}}^+ & \mbox{if no edge between $i$ and $\iota(i)$ in $\Delta$}.\\ \end{array}\right.\]

Moreover, we let 
\[e_{t_i}:=\left\{\begin{array}{ll}
 e_i & \mbox{if  $i=\iota(i)$ in $\Delta$},\\
e_i+e_{\iota(i)} & \mbox{if  $i\neq\iota(i)$ in $\Delta$}.
\end{array}\right.\]

Then, it is easy to check that
$\bmu_i^+(\La)=\mu_{(e_{t_i}\La)}^+(\La)$ and hence we have 
$$S_{t_i}=\left\{\begin{array}{cccc}
\stackrel{-1}{ e_{t_i}\La}&\stackrel{f}{\longrightarrow}&
\stackrel{0}{R_{t_i}}\\
&\oplus&&\in\Kb(\proj\La)\\
&&(1-e_{t_i})\La
\end{array}\right.$$ 
where $f$ is a minimal left $(\add((1-e_{t_i})\La))$-approximation. 

Thus $\bmu_i^+$ is an irreducible left tilting mutation of $\La$ and any irreducible left tilting mutation of $\La$ is given as $\bmu_i^+$ for some $i\in \Delta_0^{\rm f}$. Dually, we define $\bmu_i^-$ so that $\bmu_i^-\circ\bmu_i^+=\id$ (\cite[Proposition 2.33]{AI}). 

Let $F_{\Delta^{\rm f}}$ be the free group generated by $a_i$ $(i\in \Delta_0^{\rm f})$. 
%We denote by $\bmu_{a_i^+}=\bmu_{i}^{+}$ and $\bmu_{a_i^{-1}}=\bmu_{i}^{-}$. 
Then we define the map 
$$F_{\Delta^{\rm f}}\to\tilt\La,$$
$$a=a_{i_1}^{\epsilon_{i_1}}\cdots a_{i_k}^{\epsilon_{i_k}}\mapsto
\bmu_a(\La):=\bmu_{i_1}^{\epsilon_{i_1}}\circ\cdots \circ \bmu_{i_k}^{\epsilon_{i_k}}(\La).$$

Then we give the following proposition.

\begin{proposition}\label{braid relations}
For any $a\in F_{\Delta^{\rm f}}$, we let $T:=\bmu_a(\La)$.
Then we have the following braid relations in $\Db(\mod\La)$ 
\[\left\{\begin{array}{ll}
\ \bmu_i^+\circ\bmu_{j}^+(T) \cong\bmu_j^+\circ\bmu_{i}^+(T) & \mbox{if no edge between $i$ and $j$ in $\Delta^{\rm f}$},\\
\ \bmu_i^+\circ\bmu_{j}^+\circ \bmu_i^+(T)\cong\bmu_j^+\circ \bmu_i^+\circ\bmu_{j}^+(T) & \mbox{if there is an edge $i\stackrel{ }{\mbox{---}}j$ in $\Delta^{\rm f}$},\\
\ \bmu_i^+\circ\bmu_{j}^+\circ \bmu_i^+\circ\bmu_j^+(T)\cong \bmu_j^+\circ\bmu_i^+\circ\bmu_{j}^+\circ \bmu_i^+(T) & \mbox{if there is an edge $i\stackrel{4}{\mbox{---}}j$ in $\Delta^{\rm f}$}.\end{array}\right.\]
\end{proposition}

\begin{proof}
By Theorem \ref{action}, the assertion holds for $T=\La$. 
Moreover, 
by Theorem \ref{tilt discrete},  $T$ satisfies $\End_{\Kb(\proj\La)}(T)\cong\La$ and hence we have an equivalence $F:\Kb(\proj\La)\to\Kb(\proj\La)$ such that $F(T)\cong\La$. Since mutation is preserved by 
an equivalence, the assertion holds for $T$.
\end{proof}

Now we recall the following definition. 

\begin{definition}\label{braid def}
The braid group $B_{\Delta^{\rm f}}$ is defined by generators $a_i$ $(i\in \Delta^{\rm f}_0)$ and relations 
$(a_ia_j)^{m(i,j)}=1$ for $i\neq j$  (i.e. the difference with $W_{\Delta^{\rm f}}$ is that we do not require the relations $a_i^2=1$ for $i\in \Delta^{\rm f}_0$).  
Moreover we denote the positive braid monoid by $B^+_{\Delta^{\rm f}}$. 
\end{definition}

As a consequence of the above results, we have the following proposition.

\begin{proposition}\label{braid map}
There is a map 
$$B_{\Delta^{\rm f}}\to\tilt\La,\ a\mapsto \bmu_a(\La).$$ 
Moreover, it is surjective.
\end{proposition}

\begin{proof}
The first statement follows from Proposition \ref{braid relations}. 
Since $\La$ is tilting-discrete, any tilting complex can be obtain from $\La$ by iterated irreducible tilting mutation (\cite[Theorem 5.14]{CKL}, \cite[Theorem 3.5]{AI}). 
Thus the map is surjective.
\end{proof}

%%%%%%%%%%%%%%%%%%%%%%%%%%%%%%%%%%%%%%%%%%%%%%%%%%%%

Finally, we will show that the map of Proposition \ref{braid map} is injective. 

Recall that $T>\bmu_a(T)$ for any $a\in B^+_{\Delta^{\rm f}}$ (Definition \ref{defsm}). 
Then we have the following result.

\begin{lemma}\label{braid inj}
The map  
$$B_{\Delta^{\rm f}}^+\to\tilt\La,\ a\mapsto \bmu_a(\La)$$
is injective. 
\end{lemma}

\begin{proof}
We denote by $\ell(a)$ the length of $a\in B_{\Delta^{\rm f}}^+$, that is, the number of elements of the expression $a$.
We show by induction on the length of $B_{\Delta^{\rm f}}^+$.
Take $b,c\in B_{\Delta^{\rm f}}^+$ such that $\bmu_b(\La)\cong \bmu_c(\La)$ in $\Db(\mod\La)$. 
Without loss of generality, we can assume that $\ell(b)\le\ell(c)$.

If $\ell(b)=0$, (or equivalently, $b=\id$), then $\bmu_b(\La)=\La$. 
Then we have $c=\id$ because otherwise $\La>\bmu_c(\La)$. 

Next assume that $\ell(b)>0$ and the statement holds for any element if the length is less than  
$\ell(b)$. 
We write 
$b=b'a_i$ and $c=c'a_j$ for some $b',c'\in B_{\Delta^{\rm f}}^+$ and $i,j\in \Delta_0^{\rm f}$. 
If $i=j$, then $\bmu_{b'}(\La)\cong \bmu_{c'}(\La)$ and the induction 
hypothesis implies that $b'=c'$ and hence $b=c$. 

Hence assume that $i\neq j$. 
Then we define 
\[a_{i,j}:=
\left\{\begin{array}{ll}
\ a_ia_j & \mbox{if no edge between $i$ and $j$ in $\Delta^{\rm f}$},\\
\ a_ia_ja_i & \mbox{if there is an edge $i\stackrel{ }{\mbox{---}}j$ in $\Delta^{\rm f}$},\\
\  a_ia_ja_ia_j & \mbox{if there is an edge $i\stackrel{4}{\mbox{---}}j$ in $\Delta^{\rm f}$}.\end{array}\right.\]

Then $\bmu_{a_{i,j}}(\La)$ is a meet of $\bmu_{a_i}(\La)$ and $\bmu_{a_j}(\La)$ by Theorem \ref{action}, \cite[Theorem 2.30]{M1} and \cite[Corollary 3.9]{AIR}. 
Therefore we get $\bmu_{a_{i,j}}(\La)\geq \bmu_b(\La)$ since $\bmu_{a_i}(\La)\geq \bmu_b(\La)$ and $\bmu_{a_j}(\La)\geq \bmu_c(\La)\cong \bmu_b(\La)$. 

Because $\La$ is tilting-discrete and $\La>\bmu_{a_{i,j}}(\La)$, 
there exists $d\in B_{\Delta^{\rm f}}^+$ such that 
$\bmu_d(\bmu_{a_{i,j}}(\La))= \bmu_{da_{i,j}}(\La)\cong \bmu_b(\La)$. 
Then we have $\bmu_{da_{i,j}a_i^{-1}}(\La)\cong \bmu_{b'}(\La)$. 
Since we have $da_{i,j}a_i^{-1}\in B_{\Delta^{\rm f}}^+$, the induction
hypothesis implies that $da_{i,j}a_i^{-1}=b'$ and hence $da_{i,j}=b$. 
Similarly, we have $\bmu_{da_{i,j}a_j^{-1}}(\La)\cong \bmu_{c'}(\La)$ and hence we get 
$da_{i,j}a_j^{-1}=c'$. 
Therefore, we get $b=da_{i,j}=c'a_j=c$ and the assertion holds.
\end{proof}

As an immediate consequence, we obtain the following result (c.f. \cite[Lemma 2.3]{BT}).

\begin{proposition}\label{faithful}
The map  
$$B_{\Delta^{\rm f}}\to\tilt\La,\ a\mapsto \bmu_a(\La)$$
is injective. 
\end{proposition}

\begin{proof}
It is enough to show that $\bmu_a(\La)\cong \La$ in $\Db(\mod\La)$ implies 
$a=\id$. 
In fact, $\bmu_a(\La)\cong \bmu_{a'}(\La)$ implies $\bmu_{a{a'}^{-1}}(\La)\cong\La$. Then if $a{a'}^{-1}=\id$, then we get $a=a'$. 

It is well-known that any element $a\in B_{\Delta^{\rm f}}$ is given as $a=b^{-1}c$ for some $b,c\in B_{\Delta^{\rm f}}^+$ \cite[section 6.6]{KT}. 
Hence, $\bmu_a(\La)\cong \La$ is equivalent to saying that $\bmu_{b^{-1}c}(\La)\cong \La$. 
Then we have $\bmu_b(\La)\cong \bmu_c(\La)$ and Lemma \ref{braid inj} implies $b=c$. 
Thus we get the assertion.
\end{proof}

Consequently, we obtain the following conclusion.

\begin{theorem}\label{braid action}
There is a bijection 
$$B_{\Delta^{\rm f}}\to\tilt\La,\ a\mapsto \bmu_a(\La).$$
\end{theorem}

\begin{proof}
The statement follows from Propositions \ref{braid map} and \ref{faithful}.
\end{proof}

%%%%%%%%%%%%%%%%%%%%%%%%%%%%%%%%%%%%%%%%%%%%%%%%%%%%%%%%%%%%%%%%%%%%%%%%%%%%%%%%%%%%%%%%%%%%%%%%%%%%%%%%%%%%%%%%%%%%%%%%%%%%%%%%%%%%%%%%%%%%%%%%%%%%%%


\begin{thebibliography}{99}

\bibitem[AAC]{AAC}T. Adachi, T. Aihara, A. Chan,
\emph{Tilting Brauer graph algebras I: Classification of two-term tilting complexes}, arxiv:1504.04827.

\bibitem[AIR]{AIR}
T.~Adachi, O. Iyama, I. Reiten, \emph{$\tau$-tilting theory},
 Compos. Math. { 150} (2014), no. 3, 415--452.

\bibitem[A]{A}
T. Aihara,
\emph{Tilting-connected symmetric algebras},
{ Algebr. Represent. Theory} { 16} (2013), no. 3, 873--894.

\bibitem[AI]{AI}
T. Aihara, O. Iyama, 
\emph{Silting mutation in triangulated categories},
{ J. Lond. Math. Soc.} (2) {85} (2012), no. 3, 633--668.

\bibitem[AlR]{AlR}
S. Al-Nofayee, J. Rickard,
\emph{Rigidity of tilting complexes and derived equivalence for selfinjective algebras}, arxiv:1311.0504.

\bibitem[AuR]{AuR}
M. Auslander, I. Reiten,
\emph{DTr-periodic modules and functors}, Repr. theory of algebras (Cocoyoc 1994)
CMS Conf. Proc., vol. 18, Amer. Math. Soc., Providence, RI (1996), 39--50.


\bibitem[ARS]{ARS} M. Auslander, I. Reiten, S. O. Smal{\o}, \emph{Representation Theory of Artin Algebras}, Cambridge Studies in Advanced Mathematics, 36. Cambridge University Press, Cambridge, 1995.


\bibitem[BGL]{BGL} D. Baer, W. Geigle, H. Lenzing, \emph{The preprojective algebra of a tame hereditary Artin algebra}, Comm. Algebra 15 (1987), no. 1-2, 425--457.

\bibitem[BK]{BK} P. Baumann, J. Kamnitzer,  \emph{Preprojective algebras and MV polytopes}, Represent. Theory 16 (2012), 152--188.

\bibitem[BKT]{BKT} P. Baumann, J. Kamnitzer, P. Tingley, \emph{Affine Mirkovi\'c-Vilonen polytopes}, Publ. Math. Inst. Hautes Etudes Sci. 120 (2014), 113--205.


\bibitem[BB]{BB}
A.~Bj{\"o}rner, F.~Brenti, \emph{Combinatorics of {C}oxeter groups},
  Graduate Texts in Mathematics, vol. 231, Springer, New York, 2005.

\bibitem[Bo]{Bo}
K. Bongartz, \emph{Tilted algebras}, Representations of algebras (Puebla, 1980), 26--38, Lecture Notes in Math., 903, Springer, Berlin-New York, 1981.


\bibitem[BT]{BT}
C. Brav, H. Thomas,
\emph{Braid groups and Kleinian singularities},
{Math. Ann} {351} (2011), no. 4, 1005--1017.

\bibitem[BBK]{BBK}
S. Brenner, M. C. R. Butler, A. D. King, \emph{Periodic Algebras which are Almost Koszul}, Algebr. Represent.Theory 5 (2002), no. 4, 331--367.


\bibitem[BPP]{BPP}
N. Broomhead, D. Pauksztello, D. Ploog,
\emph{Discrete derived categories II:
The silting pairs CW complex and the stability manifold}, 
J. Lond. Math. Soc. (2016) 93 (2): 273-300.

\bibitem[BY]{BY} T. Br{\"u}stle, D. Yang, \emph{Ordered exchange graphs}, 
Advances in representation theory of algebras, 135--193, EMS Ser. Congr. Rep., Eur. Math. Soc., Zurich, 2013. 

\bibitem[BIRS]{BIRS}
A.~B. Buan, O.~Iyama, I.~Reiten, J.~Scott, \emph{Cluster structures for
  2-{C}alabi-{Y}au categories and unipotent groups}, Compos. Math. 145
  (2009), 1035--1079.

\bibitem[BRT]{BRT} A. B. Buan, I.  Reiten, H. Thomas,  
\emph{Three kinds of mutation}, J. Algebra 339 (2011), 97--113.



\bibitem[C]{C}{R. W. Carter},
\emph{Simple groups of Lie type},
Reprint of the 1972 original.
Wiley Classics Library.
A Wiley-Interscience Publication. 
{\it John Wiley} \& {\it Sons, Inc., New York}, 1989.

\bibitem[CKL]{CKL} A. Chan, S. Koenig, Y. Liu, \emph{Simple-minded systems, configurations and mutations for representation-finite self-injective algebras}, J. Pure Appl. Algebra, 219 (6), 2015, 1940?1961. 


\bibitem[CB]{CB}
W.~Crawley-Boevey, \emph{On the exceptional fibres of Kleinian singularities},
  Amer. J. Math.  122  (2000),  no. 5, 1027--1037.


\bibitem[DR]{DR}
V. Dlab, C. M. Ringel,  \emph{The preprojective algebra of a modulated graph}, 
Representation theory, II (Proc. Second Internat. Conf., Carleton Univ., Ottawa, Ont., 1979), pp. 216--231, Lecture Notes in Math., 832, Springer, Berlin-New York, 1980.


\bibitem[ES]{ES}{K. Erdmann, N. Snashall}, 
\emph{Preprojective algebras of Dynkin type, periodicity and the second Hochschild cohomology}, Algebras and modules, II, 183--193, CMS Conf. Proc., 24, {Amer. Math. Soc., Providence, RI}, 1998.   



\bibitem[GLS1]{GLS1}C.~Geiss, B.~Leclerc, J.~Schr{\"o}er, \emph{Rigid modules over preprojective algebras},  Invent. Math. 165 (2006), no. 3, 589--632.


\bibitem[GLS2]{GLS2}C.~Geiss, B.~Leclerc, J.~Schr{\"o}er, \emph{Kac-Moody groups and cluster algebras}, Adv. Math. 228 (2011), no. 1, 329--433.


\bibitem[GP]{GP}
I. M. Gelfand, V. A. Ponomarev, \emph{Model algebras and representations of graphs}, Funktsional. Anal. i Prilozhen. 13 (1979), no. 3, 1--12.


\bibitem[G]{G} J. Grant, \emph{Derived autoequivalences from periodic algebras}, Proc. Lond. Math. Soc. (3) 106 (2013), no. 2, 375--409.


\bibitem[IR1]{IR1}
O.~Iyama, I.~Reiten, \emph{Fomin-{Z}elevinsky mutation and tilting modules
  over {C}alabi-{Y}au algebras}, Amer. J. Math. 130 (2008), no.~4,
  1087--1149.

\bibitem[IR2]{IR2}
O.~Iyama, I.~Reiten, \emph{Introduction to $\tau$-tilting theory}, Proc. Natl. Acad. Sci. USA 111 (2014), no. 27, 9704--9711. 

\bibitem[IJY]{IJY}
O. Iyama, P. Jorgensen, D. Yang,
\emph{Intermediate co-$t$-structures, two-term silting objects, tau-tilting modules,
and torsion classes}, Algebra Number Theory 8 (2014), no. 10, 2413--2431.

\bibitem[KaS]{KaS} M. Kashiwara, Y. Saito,
\emph{Geometric construction of crystal bases}, Duke Math. J. 89 (1997), no. 1, 9--36.

\bibitem[KT]{KT} C. Kassel, V. Turaev,
\emph{Braid groups}, Graduate Texts in Mathematics, 247. Springer, New York, 2008.

\bibitem[K]{K}
B. Keller, \emph{Deriving DG categories}, 
 Ann. Sci. \'{E}cole Norm. Sup. (4)  27 (1994), no.1, 63--102.

\bibitem[KV]{KV} B. Keller, D. Vossieck, \emph{Aisles in derived categories},
Deuxieme Contact Franco-Belge en Algebre (Faulx-les-Tombes, 1987).
Bull. Soc. Math. Belg. Ser. A 40 (1988), no. 2, 239--253.


\bibitem[KhS]{KhS} M. Khovanov, P. Seidel, \emph{Quivers, Floer cohomology, and braid group actions}, J. Amer. Math. Soc. 15 (2002), no. 1, 203--271.

\bibitem[KY]{KY} S. Koenig, D. Yang, \emph{Silting objects, simple-minded collections, $t$-structures and co-$t$-structures for finite-dimensional algebras}, 
 Doc. Math. 19 (2014), 403--438.

 
\bibitem[L1]{L1} G. Lusztig, \emph{Quivers, perverse sheaves, and quantized enveloping algebras}, J. Amer. Math. Soc. 4 (1991), no. 2, 365--421.

\bibitem[L2]{L2} G. Lusztig, \emph{Semicanonical bases arising from enveloping algebras},  Adv. Math. 151 (2000), no. 2, 129--139.



\bibitem[M1]{M1}
Y. Mizuno, \emph{Classifying $\tau$-tilting modules over preprojective algebras of Dynkin type},
 Math. Z. 277 (2014), no. 3--4, 665--690.

\bibitem[M2]{M2} 
Y, Mizuno, \emph{$\nu$-stable $\tau$-tilting modules},  Comm. Algebra, 43 (2015), (4) 1654--1667.

\bibitem[N1]{N1} 
H, Nakajima, \emph{Instantons on ALE spaces, quiver varieties, and Kac-Moody algebras},  Duke Math. J. 76 (1994), no. 2, 365--416.
\bibitem[N2]{N2} 
H, Nakajima, \emph{Quiver Varieties and Kac-Moody Algebras},  Duke Math. J. 91 (1998), no. 3, 515--560.

\bibitem[N3]{N3} 
H, Nakajima, \emph{Quiver varieties and finite dimensional representations of quantum affine algebras},   J. Amer. Math. Soc. 14 (2001), no. 1, 145--238.

\bibitem[ORT]{ORT} S. Oppermann, I. Reiten, H. Thomas, \emph{Quotient closed subcategories of quiver representations}, Compositio Math.  151 (2015), 03, 568--602.


\bibitem[QW]{QW}
Y. Qiu, J. Woolf, \emph{Contractible stability spaces and faithful braid group actions},
arxiv:1407.5986.


\bibitem[Ric]{Ric} 
J. Rickard, 
\emph{Morita theory for derived categories},
{J. London Math. Soc.} (2) { 39} (1989), no. 3, 436--456. 

\bibitem[Rin]{Rin}
C. M.~Ringel, \emph{The preprojective algebra of a quiver}, 
Algebras and modules, II (Geiranger, 1996), 467--480, CMS Conf. Proc., 24, Amer. Math. Soc., Providence, RI, 1998.


\bibitem[ST]{ST}
 P. Seidel, R. Thomas,
\emph{Braid group actions on derived categories of coherent sheaves}, 
 Duke Math. J. 108 (2001), no. 1, 37--108.



\end{thebibliography}
\end{document}